\documentclass{amsart}

\usepackage{latexsym}
\usepackage[all]{xy}
\usepackage[dvipdfmx]{graphicx}
\usepackage{color}
\usepackage{pict2e}
\usepackage{epic,eepic}
\usepackage[mathscr]{eucal}
\usepackage{mathrsfs}
\def\C{\mathbb{C}}
\def\R{\mathbb{R}}
\def\Z{\mathbb{Z}}
\def\H{\mathbb{H}}
\def\P{\mathcal{P}}
\def\A{\mathcal{A}}
\def\D{\mathcal{D}}
\def\E{\mathcal{E}}
\def\cc{\mathcal{C}}
\def\z{\mathcal{Z}}
\def\stab{\mathrm{Stab}}
\def\im{\mathop{\mathcal{I}m}}
\def\CP{\mathbb{C}\mathrm{P}}
\def\T{\mathscr{T}}
\def\rep{\mathop{\rm rep}\nolimits}
\def\hom{\mathop{\rm Hom}\nolimits}
\def\ext{\mathop{\rm Ext}\nolimits}
\def\aut{\mathop{\rm Aut}\nolimits}

\usepackage{amsthm}
\theoremstyle{definition}
\newtheorem{df}{Definition}[section]
\theoremstyle{plain}
\newtheorem{thm}[df]{Theorem}
\newtheorem{prop}[df]{Proposition}
\newtheorem{lem}[df]{Lemma}
\newtheorem{cor}[df]{Corollary}

\makeatletter
\@addtoreset{equation}{section}
\makeatother

\title{The space of stability conditions for quivers with two vertices}
\author{Takahisa Shiina}
\address{Academic Support Center, Kogakuin University}
\email{kt13423@ns.kogakuin.ac.jp}

\begin{document}

\begin{abstract}
 The purpose of this article is to study the space of stability
 conditions $\stab(P_n)$ on the bounded derived category
 $\D^b(P_n)$ of finite dimensional representations of the quiver $P_n$
 with two vertices and $n$ parallel arrows.
 There is a local homeomorphism $\z:\stab(P_n)\rightarrow\C^2$.
 We show that, when the number of arrows is one or two, $\z$ is a
 covering map if we restrict it to the complement of a line arrangement.
 When the number of arrows is greater than two, we need to remove
 uncountably many lines to obtain a covering map.
\end{abstract}

\maketitle

\section{Introduction}

T.~Bridgeland introduced the notion of stability conditions on
triangulated categories (\cite{bri}).
The idea comes from Douglas's work on $\pi$-stability for D-branes in
string theory (\cite{douc}).

The main results of Bridgeland's paper are as follows:
The set of all (locally finite) stability conditions $\stab(\T)$ on a
triangulated category $\T$ has a topology.
Each connected component $\Sigma\subset\stab(\T)$ is equipped with a
local homeomorphism $\z$ to a certain topological vector space
$V(\Sigma)$ (\cite[Theorem~1.2]{bri}).
Moreover $\stab(\T)$ becomes a (possibly infinite-dimensional)
manifold.

\subsection{Background}
It is an important problem to show simply connectivity of $\Sigma$,
connectivity of $\stab(\T)$, or (universal) covering map property
of the local homeomorphism $\z:\Sigma\rightarrow V(\Sigma)$.
It have been studied intensively, in particular for
$\stab(X)=\stab(\D^b(\mathrm{coh}(X)))$, the space of stability
conditions on the bounded derived category of coherent sheaves on $X$.
For example, if $X$ is a K3 surface, Bridgeland proved that there is a
connected component $\stab^\dagger(X)\subset\stab(X)$ and a covering map
$\pi:\stab^\dagger(X)\rightarrow\P_0^+(X)$ (\cite[Theorem~1.1]{bric}).
He proved also that if $X$ is an elliptic curve, $\stab(X)$ is
connected, the image of $\z:\stab(X)\rightarrow\C^2$ is
$\mathrm{GL}^+(2,\R)$ and $\z$ is a universal covering map
(\cite[Section~9]{bri}).
E.~Macr{\`{\i}} proved that if $X$ is a curve of genus $\geq 2$,
$\stab(X)$ also is isomorphic to $\widetilde{\mathrm{GL}^+(2,\R)}$, the
universal cover of $\mathrm{GL}^+(2,\R)$ (\cite[Theorem~2.7]{mac}).
S.~Okada proved that $\stab(\mathbb{P}^1)$ is isomorphic to $\C^2$ as a
complex manifold (\cite[Theorem~1.1]{oka}).

The author was moved by Bridgeland's work on the stability conditions
for Kleinian singularities (\cite{brib}).
Let $A$ be a preprojective algebra of the ADE-graph $\Gamma$ and let
$\D=\D^b(\mathrm{mod}A)$.
He proved that there is a connected component
$\stab^\dagger(\D)\subset\stab(\D)$ which is a covering space of
$\mathfrak{h}^{\mathrm{reg}}$.
Here $\mathfrak{h}$ is the Cartan subalgebra of the simple Lie algebra
corresponding to $\Gamma$ and $\mathfrak{h}^{\mathrm{reg}}$ is the
complement of root hyperplanes in $\mathfrak{h}$.
The simply connectivity of $\stab^\dagger(\D)$ follows from
P.~Seidel and R.~Thomas \cite{st} in type A (see also \cite{tho}) and
C.~Brav and H.~Thomas \cite{bt} in type D and E.
The connectivity of $\stab(\D)$ follows from
A.~Ishii, K.~Ueda, and H.~Uehara \cite{iuu} in type A.

\subsection{Motivation and main results}
It is an important problem to find hyperplanes
$\{H_\lambda\}_{\lambda\in\Lambda}$ in $V(\Sigma)$ such that the
restriction of $\z$ to the complement
$V(\Sigma)-\cup_{\lambda\in\Lambda}H_\lambda$ is a covering map.
At first, the author thought that $\z:\Sigma\rightarrow V(\Sigma)$ was a
covering map without restricting $\Sigma$.
However the author found a inconsistency when the author examined
it for $\stab(P_n)$,
the space of stability conditions on the bounded derived category of
finite dimensional representations of a quiver $P_n$ with two vertices
and $n$ parallel arrows.

In \cite{mac}, Macr{\`{\i}} already analyzed $\stab(P_n)$; he described
coordinate neighborhoods of it, and he proved that it is a connected and
simply connected $2$-dimensional complex manifold
(\cite[Theorem~4.5]{mac}).
Although he did not refer to $\z$, it is straightforward to see the
image of $\z:\stab(P_n)\rightarrow\C^2$ is $\C^2\setminus\{(0,0)\}$ and
to see how $\z$ covers.
Nevertheless we find an interesting behavior if we collapse them by
$\C$-action.

The action of $\C$ on $\stab(P_n)$ is a part of the
$\widetilde{\mathrm{GL}^+(2,\R)}$-action
(cf. \cite{bri}, \cite[Definition~2.3]{oka}).
The quotient space is given as a corollary of \cite{mac}:

\begin{cor}
 \label{maincor}
 The quotient space $\stab(P_n)/\C$ is homeomorphic to
 \begin{itemize}
  \item $\C\setminus\{ yi\,|\,y\leq 0 \}$ if $n=1$ or
  \item $\C\setminus\bigcup_{k\in\Z}%
\left\{\left. x_k+yi\,\right|\,y\leq 0 \right\} \cup%
\left\{ x+yi\,\left|\,b_n\leq x\leq c_n,\,y<0 \right.\right\}$
       if $n\geq 2$.
 \end{itemize}
 These are the $1$-dimensional complex manifolds given by gluing three,
 if $n=1$, or countable, if $n\geq 2$, upper half planes $\H_k$.
 Here $x_k=\log\frac{a_k}{a_{k+1}}$ for all $k\in\Z$,
 $b_n=\log\frac{n-\sqrt{n^2-4}}{2}$,
 $c_n=\log\frac{n+\sqrt{n^2-4}}{2}$,
 and $\H_k$'s are copies of $\H=\{\,z\in\C\,|\,\im z>0\,\}$.
 (Figure~\ref{c1} and \ref{cn}.)
\end{cor}

\begin{figure}
{\footnotesize
\begin{center}
\begin{picture}(350,70)(0,-10)
 \dottedline{2}(30,1)(30,50)
 \put(30,0){\circle{2}}
 \dottedline{2}(0,30)(60,30)
 \dottedline{2}(0,0)(60,0)
 \put(32,32){\makebox(0,0)[bl]{$\pi i$}}
 \put(30,55){\makebox(0,0)[b]{$\H_1$}}
 \put(30,-3){\makebox(0,0)[t]{$0$}}
 \put(0,15){\line(1,0){30}}
 \qbezier(30,15)(43,13)(45,0)
 \put(45,0){\vector(1,-5){0}}%
 {\color{red}
 \qbezier(10,30)(10,22)(30,22)
 \qbezier(30,22)(50,20)(50,0)
 \put(50,0){\vector(0,-1){0}}
 }%
 {\color{magenta}
 \qbezier(45,30)(47,22)(50,20)
 \qbezier(50,20)(55,15)(55,0)
 \put(55,0){\vector(0,-1){0}}
 }%
 {\color{blue}
 \qbezier(10,0)(10,10)(30,10)
 \qbezier(30,10)(40,10)(40,0)
 \put(40,0){\vector(1,-5){0}}
 }%
 {\color{cyan}
 \qbezier(20,0)(20,5)(30,5)
 \qbezier(30,5)(35,5)(35,0)
 \put(35,0){\vector(1,-3){0}}
 }%
 \put(90,15){\vector(-1,0){20}}
 \put(80,20){\makebox(0,0)[b]{$\psi_1$}}
 \dottedline{2}(100,0)(160,0)
 \put(130,0){\circle{2}}
 \dottedline{2}(130,1)(130,50)
 \dottedline{2}(100,30)(160,30)
 \put(132,32){$\pi i$}
 \put(130,55){\makebox(0,0)[b]{$\H_0$}}
 \put(130,-3){\makebox(0,0)[t]{$0$}}
 \put(130,1){\vector(0,1){29}}%
 {\color{magenta} \put(110,0){\vector(0,1){30}}}%
 {\color{red} \put(120,0){\vector(0,1){30}}}%
 {\color{blue} \put(140,0){\vector(0,1){30}}}%
 {\color{cyan} \put(150,0){\vector(0,1){30}}}%
 \put(170,15){\vector(1,0){20}}
 \put(180,20){\makebox(0,0)[b]{$\psi_2$}}
 \dottedline{2}(200,0)(260,0)
 \put(230,0){\circle{2}}
 \dottedline{2}(200,30)(260,30)
 \dottedline{2}(230,1)(230,50)
 \put(232,32){\makebox(0,0)[lb]{$\pi i$}}
 \put(230,55){\makebox(0,0)[b]{$\H_2$}}
 \put(230,15){\line(1,0){30}}
 \qbezier(215,0)(217,13)(230,15)
 \put(215,0){\vector(-1,-5){0}}%
 {\color{blue}
 \qbezier(250,30)(250,22)(230,22)
 \qbezier(230,22)(210,20)(210,0)
 \put(210,0){\vector(0,-1){0}}
 }%
 {\color{cyan}
 \qbezier(215,30)(213,22)(210,20)
 \qbezier(210,20)(205,15)(205,0)
 \put(205,0){\vector(0,-1){0}}
 }%
 {\color{red}
 \cbezier(220,0)(220,13)(250,13)(250,0)
 \put(220,0){\vector(-1,-5){0}}
 }%
 {\color{magenta}
 \cbezier(225,0)(225,7)(240,7)(240,0)
 \put(225,0){\vector(-1,-3){0}}
 }%
 \dashline{5}(280,-10)(280,70)
 \dottedline{2}(300,25)(350,25)
 \dottedline{2}(300,40)(350,40)
 \drawline(325,0)(325,25)
 \put(325,60){\makebox(0,0){$\H_0$}}
 \put(310,0){\makebox(0,0)[b]{$\H_1$}}
 \put(340,0){\makebox(0,0)[b]{$\H_2$}}
 \drawline(295,58)(302,58)(302,65)
 \put(300,60){\makebox(0,0)[rb]{$\C$}}
 \put(325,25){\circle*{2}}
 \put(325,27){\makebox(0,0)[b]{$0$}}
 \put(325,40){\circle{2}}
 \put(325,42){\makebox(0,0)[b]{$\pi i$}}
 \multiput(305,25)(10,0){3}{\line(1,1){15}}
 \put(300,30){\line(1,1){10}}
 \put(335,25){\line(1,1){15}}
\end{picture}
\end{center}
}
\caption{$\stab(P_1)/\C$} \label{c1}
\end{figure}

\begin{figure}
{\footnotesize
\begin{center}
\begin{picture}(320,80)(0,-20)
 \dottedline{2}(0,0)(70,0) \dottedline{2}(60,0)(60,50)
 \dottedline{2}(0,30)(70,30) \put(35,45){\makebox(0,0)[b]{$\H_0$}}
 \put(20,0){\circle{2}}
 \put(25,-4){\makebox(0,0)[rt]{$\log\frac{a_{k-1}}{a_k}$}}
 \put(40,0){\circle{2}}
 \put(35,-4){\makebox(0,0)[lt]{$\log\frac{a_k}{a_{k+1}}$}}
 \put(62,2){\makebox(0,0)[lb]{$0$}}
 \put(62,32){\makebox(0,0)[lb]{$\pi i$}}
 {\color{magenta} \put(10,0){\vector(0,1){30}}}%
 {\color{red} \put(20,1){\vector(0,1){29}}}%
 \put(30,0){\vector(0,1){30}}%
 {\color{blue} \put(40,1){\vector(0,1){29}}}%
 {\color{cyan} \put(50,0){\vector(0,1){30}}}%
 \put(80,15){\vector(1,0){20}} \put(90,20){\makebox(0,0)[b]{$\psi_k$}}
 \dottedline{2}(110,0)(180,0) \dottedline{2}(110,30)(180,30)
 \dottedline{2}(120,0)(120,50)
 \put(140,45){\makebox(0,0)[b]{$\H_k$}}
 \put(118,2){\makebox(0,0)[rb]{$0$}}
 \put(118,32){\makebox(0,0)[rb]{$\pi i$}}
 \put(140,0){\circle{2}}
 \put(145,-4){\makebox(0,0)[rt]{$\log\frac{a_{k+1}}{a_k}$}}
 \put(160,0){\circle{2}}
 \put(155,-4){\makebox(0,0)[lt]{$\log\frac{a_k}{a_{k-1}}$}}
 \put(150,30){\vector(0,-1){30}}%
 {\color{red}
 \put(180,15){\line(-1,0){10}}
 \qbezier(170,15)(153,15)(153,0)
 \put(153,0){\vector(0,-1){0}}
 }%
 {\color{magenta}
 \cbezier(156,0)(156,10)(170,10)(170,0)
 \put(156,0){\vector(-1,-3){0}}
 }%
 {\color{blue}
 \put(110,15){\line(1,0){20}}
 \qbezier(130,15)(147,15)(147,0)
 \put(147,0){\vector(0,-1){0}}
 }%
 {\color{cyan}
 \qbezier(125,0)(135,14)(144,0)
 \put(144,0){\vector(1,-3){0}}
 }%
 \dashline{5}(200,-20)(200,60)
 \put(220,60){\makebox(0,0)[lt]{$(n=2)$}}
 \dottedline{2}(220,30)(320,30)
 \dottedline{2}(220,20)(320,20)
 \multiput(220,20)(10,0){10}{\line(1,1){10}}
 \drawline(270,20)(270,-20)
 \put(270,22){\makebox(0,0)[b]{$0$}}
 \put(270,32){\makebox(0,0)[b]{$\pi i$}}
 \put(270,30){\circle{2}}
 \put(270,50){\makebox(0,0){$\H_0$}}
 \drawline(230,-20)(230,20)
 \put(223,-15){\makebox(0,0){$\H_1$}}
 \drawline(250,-20)(250,20)
 \put(240,-15){\makebox(0,0){$\H_2$}}
 \put(260,-15){\makebox(0,0){$\cdots$}}
 \drawline(310,-20)(310,20)
 \put(320,-15){\makebox(0,0){$\H_{-1}$}}
 \drawline(290,-20)(290,20)
 \put(300,-15){\makebox(0,0){$\H_{-2}$}}
 \put(280,-15){\makebox(0,0){$\cdots$}}
\end{picture}
\\[3ex]
\begin{picture}(320,80)(0,-20)
 \dottedline{2}(0,0)(70,0) \dottedline{2}(0,30)(70,30)
 \dottedline{2}(10,0)(10,50) \put(35,45){\makebox(0,0)[b]{$\H_0$}}
 \put(30,0){\circle{2}}
 \put(35,-4){\makebox(0,0)[rt]{$\log\frac{a_{k+1}}{a_k}$}}
 \put(50,0){\circle{2}}
 \put(45,-4){\makebox(0,0)[lt]{$\log\frac{a_k}{a_{k-1}}$}}
 \put(8,2){\makebox(0,0)[rb]{$0$}}
 \put(8,32){\makebox(0,0)[rb]{$\pi i$}}
 {\color{magenta} \put(20,0){\vector(0,1){30}}}%
 {\color{red} \put(30,1){\vector(0,1){29}}}%
 \put(40,0){\vector(0,1){30}}%
 {\color{blue} \put(50,1){\vector(0,1){29}}}%
 {\color{cyan} \put(60,0){\vector(0,1){30}}}%
 \put(80,15){\vector(1,0){20}}
 \put(90,20){\makebox(0,0)[b]{$\psi_{-k}$}}
 \dottedline{2}(110,0)(180,0) \dottedline{2}(110,30)(180,30)
 \dottedline{2}(170,0)(170,50)
 \put(145,45){\makebox(0,0)[b]{$\H_{-k}$}}
 \put(172,2){\makebox(0,0)[lb]{$0$}}
 \put(172,32){\makebox(0,0)[lb]{$\pi i$}}
 \put(130,0){\circle{2}}
 \put(135,-4){\makebox(0,0)[rt]{$\log\frac{a_{k-1}}{a_k}$}}
 \put(150,0){\circle{2}}
 \put(145,-4){\makebox(0,0)[lt]{$\log\frac{a_k}{a_{k+1}}$}}
 \put(140,30){\vector(0,-1){30}}%
 {\color{red}
 \put(180,15){\line(-1,0){20}}
 \qbezier(160,15)(143,15)(143,0)
 \put(143,0){\vector(0,-1){0}}
 }%
 {\color{magenta}
 \cbezier(146,0)(146,10)(160,10)(160,0)
 \put(146,0){\vector(-1,-3){0}}
 }%
 {\color{blue}
 \put(110,15){\line(1,0){10}}
 \qbezier(120,15)(137,15)(137,0)
 \put(137,0){\vector(0,-1){0}}
 }%
 {\color{cyan}
 \qbezier(115,0)(125,14)(134,0)
 \put(134,0){\vector(1,-3){0}}
 }%
 \dashline{5}(200,-20)(200,60)
 \put(220,60){\makebox(0,0)[lt]{$(n>2)$}}
 \dottedline{2}(220,30)(320,30)
 \dottedline{2}(220,20)(320,20)
 \multiput(220,20)(10,0){10}{\line(1,1){10}}
 \put(270,22){\makebox(0,0)[b]{$0$}}
 \put(270,20){\circle*{2}}
 \drawline(265,-20)(265,20)(275,20)(275,-20)
 \multiput(265,-15)(0,5){8}{\line(1,0){10}}
 \put(270,32){\makebox(0,0)[b]{$\pi i$}}
 \put(270,30){\circle{2}}
 \put(270,50){\makebox(0,0){$\H_0$}}
 \drawline(230,-20)(230,20)
 \put(223,-15){\makebox(0,0){$\H_1$}}
 \drawline(250,-20)(250,20)
 \put(240,-15){\makebox(0,0){$\H_2$}}
 \put(257,-15){\makebox(0,0){$\cdots$}}
 \drawline(310,-20)(310,20)
 \put(320,-15){\makebox(0,0){$\H_{-1}$}}
 \drawline(290,-20)(290,20)
 \put(300,-15){\makebox(0,0){$\H_{-2}$}}
 \put(283,-15){\makebox(0,0){$\cdots$}}
\end{picture}
\end{center}
}
 \caption{$\stab(P_n)/\C$} \label{cn}
\end{figure}

The sequence $\{a_k\}$ (Definition \ref{defofak}) satisfies
$a_{-k}=-a_k$,
\begin{align*}
 0&=\frac{a_0}{a_1}<\frac{a_1}{a_2}<\cdots<\frac{a_{k-1}}{a_k}<\cdots
  \xrightarrow{k\to\infty} \frac{n-\sqrt{n^2-4}}{2}, \text{ and } \\
 \infty&=\frac{a_1}{a_0}>\frac{a_2}{a_1}>\cdots>\frac{a_k}{a_{k-1}}>\cdots
  \xrightarrow{k\to\infty} \frac{n+\sqrt{n^2-4}}{2}.
\end{align*}
Note that if $n=2$ then $a_k=k$ and fractions converge to $1$.
In the right side figures of Figure~\ref{c1} and \ref{cn}, illustrating
$\stab(P_n)/\C$, the solid vertical lines are $\{\,x_k+yi\,\}$
and the area
\begin{picture}(10,10)(0,2)
 \drawline(0,0)(0,10)(10,10)(10,0)
 \drawline(0,3)(10,3)
 \drawline(0,6)(10,6)
\end{picture}
($n>2$) is $\{ x+yi\,|\,b_n\leq x\leq c_n \}$.
In addition, the area
\begin{picture}(20,10)(0,2)
 \dottedline{2}(0,0)(20,0)
 \dottedline{2}(0,10)(20,10)
 \multiput(0,0)(10,0){2}{\line(1,1){10}}
\end{picture}
is the intersection of $\H_k$'s.
In the left side figures, all $\psi_k$'s, defined in Theorem
\ref{thmofccn} and \ref{defofpsik}, are homeomorphisms on
$\{ x+yi\,|\,0<y<\pi \}$.
It moves a vertical line in $\H_0$ to a curve in $\H_k$ as illustrated;
the vectors in $\H_0$ corresponds to the same colored vectors in $\H_k$
by $\psi_k$ preserving the direction.

\begin{thm}
 \label{main1}
 The local homeomorphism $\z$ induces
 $\chi_n:\stab(P_n)/\C\rightarrow\CP^1$, which
 \begin{itemize}
  \item maps the ``real axis'',
	$\{ x+0i\,|\,x\in\R \}\subset\overline{\H_0}$, to
	$\{ [1:\lambda]\,|\,\lambda>0 \}\subset\CP^1$,
  \item maps the ``$\pi$ axis'', $\{ x+\pi i\,|\,x\in\R \}\subset\H_0$, to
	$\{ [1:\mu]\,|\,\mu<0 \}\subset\CP^1$,
  \item wraps $\H_0$ on $\CP^1$ except $[0:1]$ and $[1:0]$,
  \item maps the real axis of $\overline{\H_k}$ to the longer arc from
	$[a_{k+1}:a_k]$ to $[a_k:a_{k-1}]$,
  \item maps the $\pi$ axis of $\H_k$ to the shorter arc from
	$[a_{k+1}:a_k]$ to $[a_k:a_{k-1}]$, and
  \item wraps $\H_k$ on $\CP^1$ except $[a_{k+1}:a_k]$ and
	$[a_k:a_{k-1}]$.
 \end{itemize}
 (Figure~\ref{chi1} and \ref{chi2}.)
 Moreover the minimal subset of $\CP^1$ which, after the removal,
 makes $\chi_n$ into a covering map is:
 \begin{itemize}
  \item $\{[0:1],\,[1:0],\,[1:1]\}$ if $n=1$,
  \item $\{[k:k+1]\, (k\in\Z),\,[1:1]\}$ if $n=2$, or
  \item $\left\{[a_k:a_{k+1}]\, (k\in\Z),\,[1:\lambda]\,%
\left(\frac{n-\sqrt{n^2-4}}{2}\leq\lambda\leq\frac{n+\sqrt{n^2-4}}{2}%
\right)\right\}$ if $n>2$.
 \end{itemize}
\end{thm}

\begin{figure}
{\footnotesize
\unitlength 1pt
\begin{picture}(298.4028,106.2369)(  2.6017,-113.4639)
%
\special{pn 8}%
\special{ar 638 370 60 30  3.1415927  6.2831853}%
%
%
\special{pn 8}%
\special{ar 680 370 102 42  6.2831853  3.1415927}%
%
%
\special{pn 8}%
\special{ar 638 370 138 90  3.1415927  6.2831853}%
%
%
\special{pn 8}%
\special{pa 500 362}%
\special{pa 498 370}%
\special{fp}%
\special{sh 1}%
\special{pa 498 370}%
\special{pa 528 306}%
\special{pa 506 316}%
\special{pa 488 300}%
\special{pa 498 370}%
\special{fp}%
%
%
\special{pn 8}%
\special{ar 642 832 362 360  0.0000000  6.2831853}%
%
%
\special{pn 8}%
\special{pa 280 832}%
\special{pa 282 826}%
\special{fp}%
\special{pa 302 812}%
\special{pa 308 808}%
\special{fp}%
\special{pa 334 800}%
\special{pa 340 798}%
\special{fp}%
\special{pa 366 794}%
\special{pa 374 792}%
\special{fp}%
\special{pa 400 788}%
\special{pa 408 786}%
\special{fp}%
\special{pa 434 784}%
\special{pa 440 782}%
\special{fp}%
\special{pa 468 780}%
\special{pa 474 780}%
\special{fp}%
\special{pa 502 778}%
\special{pa 508 776}%
\special{fp}%
\special{pa 536 776}%
\special{pa 542 774}%
\special{fp}%
\special{pa 570 774}%
\special{pa 578 774}%
\special{fp}%
\special{pa 604 772}%
\special{pa 612 772}%
\special{fp}%
\special{pa 638 772}%
\special{pa 646 772}%
\special{fp}%
\special{pa 672 772}%
\special{pa 680 772}%
\special{fp}%
\special{pa 706 774}%
\special{pa 714 774}%
\special{fp}%
\special{pa 740 774}%
\special{pa 748 776}%
\special{fp}%
\special{pa 776 776}%
\special{pa 782 776}%
\special{fp}%
\special{pa 808 780}%
\special{pa 816 780}%
\special{fp}%
\special{pa 842 782}%
\special{pa 850 784}%
\special{fp}%
\special{pa 876 786}%
\special{pa 884 788}%
\special{fp}%
\special{pa 910 792}%
\special{pa 916 794}%
\special{fp}%
\special{pa 944 798}%
\special{pa 950 800}%
\special{fp}%
\special{pa 976 808}%
\special{pa 982 810}%
\special{fp}%
\special{pa 1002 826}%
\special{pa 1004 832}%
\special{fp}%
%
%
\special{pn 8}%
\special{ar 642 832 362 60  6.2831853  3.1415927}%
%
%
\special{pn 20}%
\special{ar 642 832 362 360  4.7123890  1.5707963}%
%
%
\special{pn 20}%
\special{ar 642 832 482 480  5.4545864  0.8272208}%
%
%
\special{pn 8}%
\special{pa 160 592}%
\special{pa 160 1070}%
\special{fp}%
%
%
\special{pn 8}%
\special{ar 160 532 62 60  1.5707963  4.7123890}%
%
%
\special{pn 8}%
\special{ar 160 1130 62 60  1.5707963  4.7123890}%
%
%
\special{pn 8}%
\special{pa 160 472}%
\special{pa 974 472}%
\special{fp}%
%
%
\special{pn 8}%
\special{pa 160 1190}%
\special{pa 974 1190}%
\special{fp}%
%
%
\special{pn 4}%
\special{pa 100 532}%
\special{pa 100 1130}%
\special{fp}%
%
%
\special{pn 4}%
\special{pa 462 472}%
\special{pa 160 772}%
\special{fp}%
\special{pa 288 754}%
\special{pa 160 878}%
\special{fp}%
\special{pa 282 866}%
\special{pa 160 986}%
\special{fp}%
\special{pa 306 950}%
\special{pa 106 1148}%
\special{fp}%
\special{pa 336 1028}%
\special{pa 172 1190}%
\special{fp}%
\special{pa 384 1088}%
\special{pa 282 1190}%
\special{fp}%
\special{pa 444 1136}%
\special{pa 390 1190}%
\special{fp}%
\special{pa 354 472}%
\special{pa 160 664}%
\special{fp}%
\special{pa 246 472}%
\special{pa 136 580}%
\special{fp}%
%
%
\special{pn 4}%
\special{pa 1094 706}%
\special{pa 1004 796}%
\special{fp}%
\special{pa 1064 628}%
\special{pa 986 706}%
\special{fp}%
\special{pa 1028 556}%
\special{pa 950 634}%
\special{fp}%
\special{pa 974 502}%
\special{pa 902 574}%
\special{fp}%
\special{pa 896 472}%
\special{pa 842 526}%
\special{fp}%
\special{pa 1112 796}%
\special{pa 1004 902}%
\special{fp}%
\special{pa 1112 902}%
\special{pa 824 1190}%
\special{fp}%
\special{pa 1040 1082}%
\special{pa 932 1190}%
\special{fp}%
%
%
\special{pn 4}%
\special{pa 642 472}%
\special{pa 642 292}%
\special{dt 0.027}%
%
%
\special{pn 8}%
\special{pa 642 1190}%
\special{pa 642 1370}%
\special{dt 0.045}%
%
%
\special{pn 8}%
\special{ar 632 1298 60 30  6.2831853  3.1415927}%
%
%
\special{pn 8}%
\special{ar 674 1298 102 42  3.1415927  6.2831853}%
%
%
\special{pn 8}%
\special{ar 632 1298 138 90  6.2831853  3.1415927}%
%
%
\special{pn 8}%
\special{pa 494 1306}%
\special{pa 492 1298}%
\special{fp}%
\special{sh 1}%
\special{pa 492 1298}%
\special{pa 482 1368}%
\special{pa 500 1352}%
\special{pa 522 1362}%
\special{pa 492 1298}%
\special{fp}%
%
%
\special{pn 8}%
\special{ar 1848 832 362 360  0.0000000  6.2831853}%
%
%
\special{pn 8}%
\special{pa 1486 832}%
\special{pa 1488 826}%
\special{fp}%
\special{pa 1504 812}%
\special{pa 1510 810}%
\special{fp}%
\special{pa 1534 802}%
\special{pa 1540 800}%
\special{fp}%
\special{pa 1566 794}%
\special{pa 1574 792}%
\special{fp}%
\special{pa 1600 788}%
\special{pa 1608 788}%
\special{fp}%
\special{pa 1634 784}%
\special{pa 1642 782}%
\special{fp}%
\special{pa 1668 780}%
\special{pa 1676 780}%
\special{fp}%
\special{pa 1704 778}%
\special{pa 1710 776}%
\special{fp}%
\special{pa 1738 776}%
\special{pa 1746 774}%
\special{fp}%
\special{pa 1772 774}%
\special{pa 1780 774}%
\special{fp}%
\special{pa 1808 772}%
\special{pa 1816 772}%
\special{fp}%
\special{pa 1844 772}%
\special{pa 1850 772}%
\special{fp}%
\special{pa 1878 772}%
\special{pa 1886 772}%
\special{fp}%
\special{pa 1914 774}%
\special{pa 1922 774}%
\special{fp}%
\special{pa 1950 774}%
\special{pa 1956 776}%
\special{fp}%
\special{pa 1984 776}%
\special{pa 1992 778}%
\special{fp}%
\special{pa 2018 780}%
\special{pa 2026 780}%
\special{fp}%
\special{pa 2054 782}%
\special{pa 2060 784}%
\special{fp}%
\special{pa 2086 788}%
\special{pa 2094 788}%
\special{fp}%
\special{pa 2120 792}%
\special{pa 2128 794}%
\special{fp}%
\special{pa 2154 800}%
\special{pa 2160 802}%
\special{fp}%
\special{pa 2184 810}%
\special{pa 2190 812}%
\special{fp}%
\special{pa 2208 826}%
\special{pa 2210 832}%
\special{fp}%
%
%
\special{pn 8}%
\special{ar 1848 832 362 60  6.2831853  3.1415927}%
%
%
\special{pn 20}%
\special{ar 1848 832 362 360  1.5707963  6.2831853}%
%
\put(46.1805,-15){\makebox(0,0){$[0\!:\!1]$}}%
\put(46.1805,-107){\makebox(0,0){$[1\!:\!0]$}}%
%
\special{pn 8}%
\special{pa 222 532}%
\special{pa 222 292}%
\special{fp}%
\special{pa 222 292}%
\special{pa 222 292}%
\special{fp}%
%
\put(16,-16){\makebox(0,0){$\H_0$}}%
%
\special{pn 8}%
\special{ar 1820 1332 34 16  6.2831853  3.1415927}%
%
%
\special{pn 8}%
\special{ar 1844 1332 58 22  3.1415927  6.2831853}%
%
%
\special{pn 8}%
\special{ar 1820 1332 78 50  6.2831853  3.1415927}%
%
%
\special{pn 8}%
\special{pa 1744 1336}%
\special{pa 1744 1332}%
\special{fp}%
\special{sh 1}%
\special{pa 1744 1332}%
\special{pa 1724 1400}%
\special{pa 1744 1386}%
\special{pa 1764 1400}%
\special{pa 1744 1332}%
\special{fp}%
%
%
\special{pn 8}%
\special{pa 2210 832}%
\special{pa 2572 1070}%
\special{fp}%
%
%
\special{pn 8}%
\special{pa 1848 1190}%
\special{pa 2090 1550}%
\special{fp}%
%
%
\special{pn 8}%
\special{pa 2210 1430}%
\special{pa 2450 1190}%
\special{fp}%
%
%
\special{pn 8}%
\special{pa 2210 1550}%
\special{pa 2572 1190}%
\special{fp}%
%
%
\special{pn 8}%
\special{ar 2150 1490 82 82  5.4977871  2.4087776}%
%
%
\special{pn 8}%
\special{ar 2510 1130 86 84  5.4977871  2.3479302}%
%
%
\special{pn 20}%
\special{ar 1848 832 498 494  2.8975906  4.9573676}%
%
%
\special{pn 8}%
\special{pa 1968 352}%
\special{pa 2090 562}%
\special{fp}%
%
%
\special{pn 8}%
\special{pa 1366 950}%
\special{pa 1576 1070}%
\special{fp}%
%
%
\special{pn 4}%
\special{pa 1534 460}%
\special{pa 1534 640}%
\special{fp}%
\special{pa 1462 538}%
\special{pa 1462 1004}%
\special{fp}%
\special{pa 1390 676}%
\special{pa 1390 962}%
\special{fp}%
\special{pa 1534 1022}%
\special{pa 1534 1046}%
\special{fp}%
\special{pa 1606 412}%
\special{pa 1606 556}%
\special{fp}%
\special{pa 1680 376}%
\special{pa 1680 508}%
\special{fp}%
\special{pa 1752 358}%
\special{pa 1752 478}%
\special{fp}%
\special{pa 1824 346}%
\special{pa 1824 466}%
\special{fp}%
\special{pa 1896 346}%
\special{pa 1896 472}%
\special{fp}%
\special{pa 1968 358}%
\special{pa 1968 490}%
\special{fp}%
\special{pa 2040 478}%
\special{pa 2040 520}%
\special{fp}%
%
%
\special{pn 4}%
\special{pa 2258 860}%
\special{pa 2258 1382}%
\special{fp}%
\special{pa 2186 962}%
\special{pa 2186 1556}%
\special{fp}%
\special{pa 2114 1076}%
\special{pa 2114 1556}%
\special{fp}%
\special{pa 2040 1136}%
\special{pa 2040 1478}%
\special{fp}%
\special{pa 1968 1166}%
\special{pa 1968 1370}%
\special{fp}%
\special{pa 1896 1184}%
\special{pa 1896 1262}%
\special{fp}%
\special{pa 2330 908}%
\special{pa 2330 1310}%
\special{fp}%
\special{pa 2402 956}%
\special{pa 2402 1238}%
\special{fp}%
\special{pa 2474 1004}%
\special{pa 2474 1208}%
\special{fp}%
\special{pa 2546 1052}%
\special{pa 2546 1202}%
\special{fp}%
%
%
\special{pn 8}%
\special{pa 2210 832}%
\special{pa 2450 832}%
\special{dt 0.045}%
%
%
\special{pn 8}%
\special{pa 1848 1190}%
\special{pa 1848 1370}%
\special{dt 0.045}%
%
%
\special{pn 8}%
\special{pa 2408 866}%
\special{pa 2422 842}%
\special{pa 2420 810}%
\special{pa 2408 800}%
\special{sp}%
%
%
\special{pn 8}%
\special{pa 2408 800}%
\special{pa 2388 822}%
\special{pa 2384 854}%
\special{pa 2388 886}%
\special{pa 2404 912}%
\special{pa 2408 912}%
\special{sp}%
%
%
\special{pn 8}%
\special{pa 2408 908}%
\special{pa 2436 894}%
\special{pa 2452 868}%
\special{pa 2456 836}%
\special{pa 2452 804}%
\special{pa 2440 776}%
\special{pa 2414 758}%
\special{pa 2408 758}%
\special{sp}%
%
%
\special{pn 8}%
\special{pa 2412 758}%
\special{pa 2408 758}%
\special{fp}%
\special{sh 1}%
\special{pa 2408 758}%
\special{pa 2474 778}%
\special{pa 2460 758}%
\special{pa 2474 738}%
\special{pa 2408 758}%
\special{fp}%
%
\put(133.4827,-107){\makebox(0,0){$[1\!:\!0]$}}%
\put(181.3977,-60){\makebox(0,0)[l]{$[1\!:\!1]$}}%
%
\special{pn 8}%
\special{pa 1576 472}%
\special{pa 1576 292}%
\special{fp}%
%
\put(113.8975,-16.7666){\makebox(0,0){$\H_1$}}%
%
\special{pn 8}%
\special{ar 3418 840 362 360  0.0000000  6.2831853}%
%
%
\special{pn 8}%
\special{ar 3418 840 362 60  6.2831853  3.1415927}%
%
%
\special{pn 8}%
\special{pa 3056 840}%
\special{pa 3058 834}%
\special{fp}%
\special{pa 3074 820}%
\special{pa 3080 818}%
\special{fp}%
\special{pa 3100 810}%
\special{pa 3108 810}%
\special{fp}%
\special{pa 3130 804}%
\special{pa 3138 802}%
\special{fp}%
\special{pa 3162 798}%
\special{pa 3168 796}%
\special{fp}%
\special{pa 3194 794}%
\special{pa 3202 792}%
\special{fp}%
\special{pa 3226 790}%
\special{pa 3234 788}%
\special{fp}%
\special{pa 3260 786}%
\special{pa 3268 786}%
\special{fp}%
\special{pa 3294 784}%
\special{pa 3302 784}%
\special{fp}%
\special{pa 3328 782}%
\special{pa 3336 782}%
\special{fp}%
\special{pa 3362 782}%
\special{pa 3370 782}%
\special{fp}%
\special{pa 3398 780}%
\special{pa 3404 780}%
\special{fp}%
\special{pa 3432 780}%
\special{pa 3440 780}%
\special{fp}%
\special{pa 3466 782}%
\special{pa 3474 782}%
\special{fp}%
\special{pa 3500 782}%
\special{pa 3508 782}%
\special{fp}%
\special{pa 3534 784}%
\special{pa 3542 784}%
\special{fp}%
\special{pa 3568 786}%
\special{pa 3576 786}%
\special{fp}%
\special{pa 3602 788}%
\special{pa 3610 790}%
\special{fp}%
\special{pa 3636 792}%
\special{pa 3642 794}%
\special{fp}%
\special{pa 3668 796}%
\special{pa 3674 798}%
\special{fp}%
\special{pa 3698 802}%
\special{pa 3706 804}%
\special{fp}%
\special{pa 3728 810}%
\special{pa 3734 812}%
\special{fp}%
\special{pa 3758 820}%
\special{pa 3762 822}%
\special{fp}%
\special{pa 3778 834}%
\special{pa 3780 840}%
\special{fp}%
%
%
\special{pn 20}%
\special{ar 3418 840 362 360  6.2831853  4.7123890}%
%
%
\special{pn 8}%
\special{ar 3390 338 34 16  3.1415927  6.2831853}%
%
%
\special{pn 8}%
\special{ar 3412 338 56 22  6.2831853  3.1415927}%
%
%
\special{pn 8}%
\special{ar 3390 338 76 50  3.1415927  6.2831853}%
%
%
\special{pn 8}%
\special{pa 3314 334}%
\special{pa 3314 338}%
\special{fp}%
\special{sh 1}%
\special{pa 3314 338}%
\special{pa 3334 272}%
\special{pa 3314 286}%
\special{pa 3294 272}%
\special{pa 3314 338}%
\special{fp}%
%
%
\special{pn 8}%
\special{pa 3780 840}%
\special{pa 4140 600}%
\special{fp}%
%
%
\special{pn 8}%
\special{pa 3418 480}%
\special{pa 3658 122}%
\special{fp}%
%
%
\special{pn 8}%
\special{pa 3780 242}%
\special{pa 4020 480}%
\special{fp}%
%
%
\special{pn 8}%
\special{pa 3780 122}%
\special{pa 4140 480}%
\special{fp}%
%
%
\special{pn 8}%
\special{ar 3720 182 82 82  3.8661963  0.7853982}%
%
%
\special{pn 8}%
\special{ar 4080 540 86 86  3.9269908  0.7853982}%
%
%
\special{pn 20}%
\special{ar 3418 840 498 494  1.3253306  3.3855948}%
%
%
\special{pn 8}%
\special{pa 3538 1320}%
\special{pa 3658 1110}%
\special{fp}%
%
%
\special{pn 8}%
\special{pa 2936 720}%
\special{pa 3146 600}%
\special{fp}%
%
%
\special{pn 4}%
\special{pa 3104 1212}%
\special{pa 3104 1032}%
\special{fp}%
\special{pa 3032 1134}%
\special{pa 3032 666}%
\special{fp}%
\special{pa 2960 996}%
\special{pa 2960 708}%
\special{fp}%
\special{pa 3104 648}%
\special{pa 3104 624}%
\special{fp}%
\special{pa 3176 1260}%
\special{pa 3176 1116}%
\special{fp}%
\special{pa 3248 1296}%
\special{pa 3248 1164}%
\special{fp}%
\special{pa 3322 1314}%
\special{pa 3322 1194}%
\special{fp}%
\special{pa 3394 1326}%
\special{pa 3394 1206}%
\special{fp}%
\special{pa 3466 1326}%
\special{pa 3466 1200}%
\special{fp}%
\special{pa 3538 1314}%
\special{pa 3538 1182}%
\special{fp}%
\special{pa 3610 1194}%
\special{pa 3610 1152}%
\special{fp}%
%
%
\special{pn 4}%
\special{pa 3828 810}%
\special{pa 3828 290}%
\special{fp}%
\special{pa 3756 708}%
\special{pa 3756 116}%
\special{fp}%
\special{pa 3682 594}%
\special{pa 3682 116}%
\special{fp}%
\special{pa 3610 534}%
\special{pa 3610 194}%
\special{fp}%
\special{pa 3538 504}%
\special{pa 3538 302}%
\special{fp}%
\special{pa 3466 486}%
\special{pa 3466 408}%
\special{fp}%
\special{pa 3900 762}%
\special{pa 3900 360}%
\special{fp}%
\special{pa 3972 714}%
\special{pa 3972 432}%
\special{fp}%
\special{pa 4044 666}%
\special{pa 4044 462}%
\special{fp}%
\special{pa 4116 618}%
\special{pa 4116 468}%
\special{fp}%
%
%
\special{pn 8}%
\special{pa 3780 840}%
\special{pa 4020 840}%
\special{dt 0.045}%
%
%
\special{pn 8}%
\special{pa 3418 480}%
\special{pa 3418 302}%
\special{dt 0.045}%
%
%
\special{pn 8}%
\special{pa 3976 806}%
\special{pa 3992 828}%
\special{pa 3990 860}%
\special{pa 3976 872}%
\special{sp}%
%
%
\special{pn 8}%
\special{pa 3976 872}%
\special{pa 3958 848}%
\special{pa 3954 816}%
\special{pa 3958 784}%
\special{pa 3974 760}%
\special{pa 3976 760}%
\special{sp}%
%
%
\special{pn 8}%
\special{pa 3976 762}%
\special{pa 4004 776}%
\special{pa 4020 802}%
\special{pa 4026 834}%
\special{pa 4022 866}%
\special{pa 4010 894}%
\special{pa 3984 914}%
\special{pa 3976 914}%
\special{sp}%
%
%
\special{pn 8}%
\special{pa 3982 914}%
\special{pa 3976 914}%
\special{fp}%
\special{sh 1}%
\special{pa 3976 914}%
\special{pa 4044 934}%
\special{pa 4030 914}%
\special{pa 4044 894}%
\special{pa 3976 914}%
\special{fp}%
%
%
\special{pn 8}%
\special{pa 3144 1190}%
\special{pa 3144 1370}%
\special{fp}%
%
\put(295.0784,-60){\makebox(0,0)[l]{$[1\!:\!1]$}}%
\put(247,-13){\makebox(0,0){$[0\!:\!1]$}}%
\put(228,-104){\makebox(0,0){$\H_2$}}%
\end{picture}%
}
 \caption{$\chi_1$} \label{chi1}
\end{figure}

\begin{figure}
{\footnotesize
\unitlength 1pt
\begin{picture}(297.1742,108.4050)(  2.5294,-115.6320)
%
\special{pn 8}%
\special{ar 634 386 60 30  3.1415927  6.2831853}%
%
%
\special{pn 8}%
\special{ar 676 386 102 42  6.2831853  3.1415927}%
%
%
\special{pn 8}%
\special{ar 634 386 138 90  3.1415927  6.2831853}%
%
%
\special{pn 8}%
\special{pa 498 380}%
\special{pa 496 386}%
\special{fp}%
\special{sh 1}%
\special{pa 496 386}%
\special{pa 526 324}%
\special{pa 504 334}%
\special{pa 486 318}%
\special{pa 496 386}%
\special{fp}%
%
%
\special{pn 8}%
\special{ar 640 846 360 358  0.0000000  6.2831853}%
%
%
\special{pn 8}%
\special{pa 280 846}%
\special{pa 282 840}%
\special{fp}%
\special{pa 300 826}%
\special{pa 308 824}%
\special{fp}%
\special{pa 332 816}%
\special{pa 338 814}%
\special{fp}%
\special{pa 364 808}%
\special{pa 372 808}%
\special{fp}%
\special{pa 398 802}%
\special{pa 406 802}%
\special{fp}%
\special{pa 432 798}%
\special{pa 438 798}%
\special{fp}%
\special{pa 466 794}%
\special{pa 472 794}%
\special{fp}%
\special{pa 500 792}%
\special{pa 506 792}%
\special{fp}%
\special{pa 534 790}%
\special{pa 540 790}%
\special{fp}%
\special{pa 568 788}%
\special{pa 574 788}%
\special{fp}%
\special{pa 602 788}%
\special{pa 608 788}%
\special{fp}%
\special{pa 636 788}%
\special{pa 642 788}%
\special{fp}%
\special{pa 670 788}%
\special{pa 678 788}%
\special{fp}%
\special{pa 704 788}%
\special{pa 712 788}%
\special{fp}%
\special{pa 738 790}%
\special{pa 746 790}%
\special{fp}%
\special{pa 772 792}%
\special{pa 780 792}%
\special{fp}%
\special{pa 806 794}%
\special{pa 814 794}%
\special{fp}%
\special{pa 840 798}%
\special{pa 846 798}%
\special{fp}%
\special{pa 874 802}%
\special{pa 880 802}%
\special{fp}%
\special{pa 906 808}%
\special{pa 914 808}%
\special{fp}%
\special{pa 940 814}%
\special{pa 946 816}%
\special{fp}%
\special{pa 972 824}%
\special{pa 978 826}%
\special{fp}%
\special{pa 998 842}%
\special{pa 1000 846}%
\special{fp}%
%
%
\special{pn 8}%
\special{ar 640 846 360 60  6.2831853  3.1415927}%
%
%
\special{pn 20}%
\special{ar 640 846 360 358  4.7123890  1.5707963}%
%
%
\special{pn 20}%
\special{ar 640 846 480 478  5.4543362  0.8274633}%
%
%
\special{pn 8}%
\special{pa 160 608}%
\special{pa 160 1084}%
\special{fp}%
%
%
\special{pn 8}%
\special{ar 160 548 60 60  1.5707963  4.7123890}%
%
%
\special{pn 8}%
\special{ar 160 1144 60 60  1.5707963  4.7123890}%
%
%
\special{pn 8}%
\special{pa 160 488}%
\special{pa 970 488}%
\special{fp}%
%
%
\special{pn 8}%
\special{pa 160 1204}%
\special{pa 970 1204}%
\special{fp}%
%
%
\special{pn 4}%
\special{pa 100 548}%
\special{pa 100 1144}%
\special{fp}%
%
%
\special{pn 4}%
\special{pa 460 488}%
\special{pa 160 786}%
\special{fp}%
\special{pa 286 768}%
\special{pa 160 894}%
\special{fp}%
\special{pa 280 882}%
\special{pa 160 1002}%
\special{fp}%
\special{pa 304 966}%
\special{pa 106 1162}%
\special{fp}%
\special{pa 334 1042}%
\special{pa 172 1204}%
\special{fp}%
\special{pa 382 1102}%
\special{pa 280 1204}%
\special{fp}%
\special{pa 442 1150}%
\special{pa 388 1204}%
\special{fp}%
\special{pa 352 488}%
\special{pa 160 680}%
\special{fp}%
\special{pa 244 488}%
\special{pa 136 596}%
\special{fp}%
%
%
\special{pn 4}%
\special{pa 1090 720}%
\special{pa 1000 810}%
\special{fp}%
\special{pa 1060 644}%
\special{pa 982 720}%
\special{fp}%
\special{pa 1024 572}%
\special{pa 946 650}%
\special{fp}%
\special{pa 970 518}%
\special{pa 898 590}%
\special{fp}%
\special{pa 892 488}%
\special{pa 838 542}%
\special{fp}%
\special{pa 1108 810}%
\special{pa 1000 918}%
\special{fp}%
\special{pa 1108 918}%
\special{pa 820 1204}%
\special{fp}%
\special{pa 1036 1096}%
\special{pa 928 1204}%
\special{fp}%
%
%
\special{pn 4}%
\special{pa 640 488}%
\special{pa 640 310}%
\special{dt 0.027}%
%
%
\special{pn 8}%
\special{pa 640 1204}%
\special{pa 640 1382}%
\special{dt 0.045}%
%
%
\special{pn 8}%
\special{ar 628 1312 60 30  6.2831853  3.1415927}%
%
%
\special{pn 8}%
\special{ar 670 1312 102 42  3.1415927  6.2831853}%
%
%
\special{pn 8}%
\special{ar 628 1312 138 90  6.2831853  3.1415927}%
%
%
\special{pn 8}%
\special{pa 492 1318}%
\special{pa 490 1312}%
\special{fp}%
\special{sh 1}%
\special{pa 490 1312}%
\special{pa 480 1380}%
\special{pa 498 1364}%
\special{pa 520 1374}%
\special{pa 490 1312}%
\special{fp}%
%
\put(46.1805,-15){\makebox(0,0){$[0\!:\!1]$}}%
\put(46.1805,-107){\makebox(0,0){$[1\!:\!0]$}}%
%
\special{pn 8}%
\special{pa 220 548}%
\special{pa 220 310}%
\special{fp}%
\special{pa 220 310}%
\special{pa 220 310}%
\special{fp}%
%
\put(15.8994,-17.9952){\makebox(0,0){$\H_0$}}%
%
\special{pn 8}%
\special{ar 3400 850 360 358  0.0000000  6.2831853}%
%
%
\special{pn 8}%
\special{ar 3400 850 360 60  6.2831853  3.1415927}%
%
%
\special{pn 8}%
\special{pa 3040 850}%
\special{pa 3042 844}%
\special{fp}%
\special{pa 3058 830}%
\special{pa 3064 828}%
\special{fp}%
\special{pa 3086 820}%
\special{pa 3092 818}%
\special{fp}%
\special{pa 3116 812}%
\special{pa 3122 812}%
\special{fp}%
\special{pa 3146 806}%
\special{pa 3154 806}%
\special{fp}%
\special{pa 3180 802}%
\special{pa 3186 802}%
\special{fp}%
\special{pa 3212 798}%
\special{pa 3218 798}%
\special{fp}%
\special{pa 3244 796}%
\special{pa 3252 794}%
\special{fp}%
\special{pa 3278 794}%
\special{pa 3286 792}%
\special{fp}%
\special{pa 3312 792}%
\special{pa 3320 792}%
\special{fp}%
\special{pa 3346 790}%
\special{pa 3354 790}%
\special{fp}%
\special{pa 3380 790}%
\special{pa 3388 790}%
\special{fp}%
\special{pa 3414 790}%
\special{pa 3422 790}%
\special{fp}%
\special{pa 3448 790}%
\special{pa 3456 790}%
\special{fp}%
\special{pa 3482 792}%
\special{pa 3490 792}%
\special{fp}%
\special{pa 3516 792}%
\special{pa 3524 794}%
\special{fp}%
\special{pa 3550 794}%
\special{pa 3558 796}%
\special{fp}%
\special{pa 3584 798}%
\special{pa 3590 798}%
\special{fp}%
\special{pa 3616 802}%
\special{pa 3624 802}%
\special{fp}%
\special{pa 3648 806}%
\special{pa 3654 808}%
\special{fp}%
\special{pa 3680 812}%
\special{pa 3686 812}%
\special{fp}%
\special{pa 3710 818}%
\special{pa 3716 820}%
\special{fp}%
\special{pa 3736 828}%
\special{pa 3742 830}%
\special{fp}%
\special{pa 3758 844}%
\special{pa 3760 850}%
\special{fp}%
%
%
\special{pn 20}%
\special{ar 3398 840 360 358  5.5586157  5.1913026}%
%
%
\special{pn 20}%
\special{ar 3338 900 420 418  0.7338974  4.2879162}%
%
%
\special{pn 8}%
\special{pa 3164 518}%
\special{pa 3284 500}%
\special{fp}%
\special{pa 3644 1180}%
\special{pa 3668 1080}%
\special{fp}%
%
%
\special{pn 8}%
\special{pa 3560 530}%
\special{pa 3524 124}%
\special{fp}%
\special{pa 3662 602}%
\special{pa 4122 722}%
\special{fp}%
%
%
\special{pn 8}%
\special{ar 3584 184 84 84  3.9269908  0.7853982}%
%
%
\special{pn 8}%
\special{ar 4062 662 86 84  3.9269908  0.7936624}%
%
%
\special{pn 8}%
\special{pa 3638 244}%
\special{pa 4008 608}%
\special{fp}%
%
%
\special{pn 4}%
\special{pa 3758 362}%
\special{pa 3758 626}%
\special{fp}%
\special{pa 3828 428}%
\special{pa 3828 644}%
\special{fp}%
\special{pa 3900 500}%
\special{pa 3900 662}%
\special{fp}%
\special{pa 3972 572}%
\special{pa 3972 684}%
\special{fp}%
\special{pa 4044 584}%
\special{pa 4044 702}%
\special{fp}%
\special{pa 4116 596}%
\special{pa 4116 720}%
\special{fp}%
\special{pa 3686 292}%
\special{pa 3686 608}%
\special{fp}%
\special{pa 3614 106}%
\special{pa 3614 560}%
\special{fp}%
\special{pa 3542 112}%
\special{pa 3542 298}%
\special{fp}%
%
%
\special{pn 8}%
\special{pa 3644 130}%
\special{pa 4122 596}%
\special{fp}%
%
%
\special{pn 4}%
\special{pa 3110 1066}%
\special{pa 3110 1240}%
\special{fp}%
\special{pa 3038 918}%
\special{pa 3038 1180}%
\special{fp}%
\special{pa 2966 726}%
\special{pa 2966 1072}%
\special{fp}%
\special{pa 3038 620}%
\special{pa 3038 762}%
\special{fp}%
\special{pa 3110 560}%
\special{pa 3110 614}%
\special{fp}%
\special{pa 3182 1132}%
\special{pa 3182 1282}%
\special{fp}%
\special{pa 3254 1174}%
\special{pa 3254 1300}%
\special{fp}%
\special{pa 3326 1198}%
\special{pa 3326 1312}%
\special{fp}%
\special{pa 3398 1204}%
\special{pa 3398 1306}%
\special{fp}%
\special{pa 3470 1198}%
\special{pa 3470 1288}%
\special{fp}%
\special{pa 3542 1174}%
\special{pa 3542 1252}%
\special{fp}%
\special{pa 3614 1132}%
\special{pa 3614 1204}%
\special{fp}%
%
%
\special{pn 8}%
\special{pa 3158 1204}%
\special{pa 3158 1382}%
\special{fp}%
%
\put(228.1564,-104.2133){\makebox(0,0){$H_{-k}$}}%
%
\special{pn 8}%
\special{ar 3398 536 156 192  4.8533400  6.2831853}%
%
\put(245.0676,-24.9332){\makebox(0,0)[rb]{$[a_{k-1}\!:\!a_k]$}}%
%
\special{pn 8}%
\special{ar 3662 786 168 180  4.7319943  6.2831853}%
%
\put(276.6496,-58.9723){\makebox(0,0)[lt]{$[a_k\!:\!a_{k+1}]$}}%
%
\special{pn 8}%
\special{ar 1842 852 360 358  0.0000000  6.2831853}%
%
%
\special{pn 8}%
\special{pa 1482 852}%
\special{pa 1484 846}%
\special{fp}%
\special{pa 1502 832}%
\special{pa 1508 830}%
\special{fp}%
\special{pa 1530 822}%
\special{pa 1538 820}%
\special{fp}%
\special{pa 1562 814}%
\special{pa 1570 812}%
\special{fp}%
\special{pa 1596 808}%
\special{pa 1602 806}%
\special{fp}%
\special{pa 1630 804}%
\special{pa 1636 802}%
\special{fp}%
\special{pa 1664 800}%
\special{pa 1670 798}%
\special{fp}%
\special{pa 1698 796}%
\special{pa 1704 796}%
\special{fp}%
\special{pa 1732 794}%
\special{pa 1740 794}%
\special{fp}%
\special{pa 1768 792}%
\special{pa 1774 792}%
\special{fp}%
\special{pa 1802 792}%
\special{pa 1810 792}%
\special{fp}%
\special{pa 1838 792}%
\special{pa 1844 792}%
\special{fp}%
\special{pa 1872 792}%
\special{pa 1880 792}%
\special{fp}%
\special{pa 1908 792}%
\special{pa 1914 792}%
\special{fp}%
\special{pa 1942 794}%
\special{pa 1950 794}%
\special{fp}%
\special{pa 1976 796}%
\special{pa 1984 796}%
\special{fp}%
\special{pa 2012 798}%
\special{pa 2018 800}%
\special{fp}%
\special{pa 2046 802}%
\special{pa 2054 802}%
\special{fp}%
\special{pa 2080 806}%
\special{pa 2086 808}%
\special{fp}%
\special{pa 2112 812}%
\special{pa 2120 814}%
\special{fp}%
\special{pa 2146 820}%
\special{pa 2152 820}%
\special{fp}%
\special{pa 2176 830}%
\special{pa 2182 832}%
\special{fp}%
\special{pa 2200 846}%
\special{pa 2202 852}%
\special{fp}%
%
%
\special{pn 8}%
\special{ar 1842 852 360 60  6.2831853  3.1415927}%
%
%
\special{pn 20}%
\special{ar 1838 860 360 358  1.0918827  0.7245696}%
%
%
\special{pn 20}%
\special{ar 1778 800 420 418  1.9940833  5.5510929}%
%
%
\special{pn 8}%
\special{pa 1606 1182}%
\special{pa 1724 1200}%
\special{fp}%
\special{pa 2084 520}%
\special{pa 2108 622}%
\special{fp}%
%
%
\special{pn 8}%
\special{pa 2000 1170}%
\special{pa 1964 1576}%
\special{fp}%
\special{pa 2102 1098}%
\special{pa 2564 980}%
\special{fp}%
%
%
\special{pn 8}%
\special{ar 2024 1516 86 84  5.4895229  2.3479302}%
%
%
\special{pn 8}%
\special{ar 2504 1040 84 84  5.4977871  2.3645977}%
%
%
\special{pn 8}%
\special{pa 2078 1458}%
\special{pa 2450 1094}%
\special{fp}%
%
%
\special{pn 4}%
\special{pa 2198 1338}%
\special{pa 2198 1076}%
\special{fp}%
\special{pa 2270 1272}%
\special{pa 2270 1058}%
\special{fp}%
\special{pa 2342 1200}%
\special{pa 2342 1040}%
\special{fp}%
\special{pa 2414 1130}%
\special{pa 2414 1016}%
\special{fp}%
\special{pa 2486 1118}%
\special{pa 2486 998}%
\special{fp}%
\special{pa 2558 1106}%
\special{pa 2558 980}%
\special{fp}%
\special{pa 2126 1410}%
\special{pa 2126 1094}%
\special{fp}%
\special{pa 2054 1594}%
\special{pa 2054 1142}%
\special{fp}%
\special{pa 1982 1588}%
\special{pa 1982 1404}%
\special{fp}%
%
%
\special{pn 8}%
\special{pa 2084 1570}%
\special{pa 2564 1106}%
\special{fp}%
%
%
\special{pn 4}%
\special{pa 1552 634}%
\special{pa 1552 462}%
\special{fp}%
\special{pa 1480 784}%
\special{pa 1480 520}%
\special{fp}%
\special{pa 1408 974}%
\special{pa 1408 628}%
\special{fp}%
\special{pa 1480 1082}%
\special{pa 1480 938}%
\special{fp}%
\special{pa 1552 1142}%
\special{pa 1552 1088}%
\special{fp}%
\special{pa 1624 568}%
\special{pa 1624 420}%
\special{fp}%
\special{pa 1694 526}%
\special{pa 1694 402}%
\special{fp}%
\special{pa 1766 504}%
\special{pa 1766 390}%
\special{fp}%
\special{pa 1838 498}%
\special{pa 1838 396}%
\special{fp}%
\special{pa 1910 504}%
\special{pa 1910 414}%
\special{fp}%
\special{pa 1982 526}%
\special{pa 1982 450}%
\special{fp}%
\special{pa 2054 568}%
\special{pa 2054 498}%
\special{fp}%
%
%
\special{pn 8}%
\special{pa 1600 498}%
\special{pa 1600 318}%
\special{fp}%
%
%
\special{pn 8}%
\special{ar 1838 1166 156 192  6.2831853  1.4298453}%
%
%
\special{pn 8}%
\special{ar 2102 914 168 180  6.2831853  1.5513183}%
%
\put(115.5597,-17.9952){\makebox(0,0){$\H_k$}}%
\put(132.8323,-97.3477){\makebox(0,0)[rt]{$[a_k\!:\!a_{k-1}]$}}%
\put(163.1857,-64.5371){\makebox(0,0)[lb]{$[a_{k+1}\!:\!a_k]$}}%
\end{picture}%
}
 \caption{$\chi_n$}  \label{chi2}
\end{figure}

\begin{cor}
 \label{main2}
 The local homeomorphism
 $\z:\stab(P_n)\rightarrow\C^2\setminus\{(0,0)\}$ is a covering map
 if it is restricted to the complement of three lines ($n=1$),
 countably many lines ($n=2$), or a bundle of uncountably many lines
 ($n>2$).
\end{cor}

\subsection{Organization of this article}
In section~2, we recall the notion of stability conditions on a
triangulated category, recall the action of
$\widetilde{\mathrm{GL}^+(2,\R)}$,
and recall Macr{\`{\i}}'s work for a triangulated category generated by
finitely many exceptional objects.
In section~3, we recall Macr{\`{\i}}'s work describing coordinate
neighborhoods of $\stab(P_n)$.
In section~4 and section~5, we see how the coordinate neighborhoods are
glued together and prove Corollary~\ref{maincor}.
In section 6 we prove Theorem~\ref{main1} and Corollary~\ref{main2}.

\subsection*{Acknowledgements}
I would like to thank Dai Tamaki for introducing me to the subject and
invaluable advice.

\section{Stability conditions}

In this section we assume that a triangulated category $\T$ satisfies
the following conditions.
\begin{itemize}
 \item $\hom^\bullet_\T(A,B)=\bigoplus_k\hom^k_\T(A,B)%
        =\bigoplus_k\hom_\T(A,B[k])$
       is a finite dimensional $\C$-vector space where $[k]$ is the
       $k$-fold shift functor.
 \item $\T$ is essentially small, i.e. $\T$ is equivalent to a category
       in which the class of objects is a set.
\end{itemize}

The Grothendieck group of $\T$, $K(\T)$, is the quotient group of the
free abelian group generated by all isomorphism classes of objects in
$\T$ over the subgroup generated by the elements of the form
$[A]+[B]-[C]$ for each distinguished triangle
$A \rightarrow C \rightarrow B$ in $\T$.

\subsection{Definition of stability conditions}

A \emph{stability condition} on $\T$ is a pair $\sigma=(Z,\P)$ of
\begin{itemize}
 \item a group homomorphism $Z:K(\T) \rightarrow \C$ and
 \item full additive subcategories $\P(\phi)$ of $\T$ for each
       $\phi \in \R$
\end{itemize}
satisfying
\begin{description}
 \item[Br1] if $\phi_1>\phi_2$ and $A_j \in \P(\phi_j)$ then
	    $\hom_\T(A_1,A_2)=0$,
 \item[Br2] $\P(\phi+1)=\P(\phi)[1]$,
 \item[Br3] each nonzero object has a
	    \emph{Harder-Narasimhan filtration}, and
 \item[Br4] if $0 \neq A \in \P(\phi)$ then
	    $Z([A])=m(A)\, e^{i\pi\phi}$ for some
	    $m(A) \in \R_{>0}$.
\end{description}
A Harder-Narasimhan filtration of a nonzero object $E\in\T$ is a diagram
\[
 \xymatrix{
  0=E_0 \ar[rr] && E_1 \ar[r] \ar[dl] & \dots \ar[r] & 
  E_{n-1} \ar[rr] && E_n=E \ar[dl] \\
  & A_1 \ar@{..>}[lu] &&&& A_n \ar@{..>}[lu] &
 }
\]
such that each $E_{j-1}\rightarrow E_{j}\rightarrow A_{j}$ is a
distinguished triangle in $\T$, $A_j\in\P(\phi_j)$ for all $j$ and
$\phi_1>\dots>\phi_n$.
Note that for a stability condition $\sigma=(Z,\P)$;
\begin{itemize}
 \item a Harder-Narasimhan filtration for a nonzero object is unique up
       to isomorphism,
 \item $Z$ is called the \emph{central charge},
 \item nonzero objects in $\P(\phi)$ are called
       \emph{semistable of phase} $\phi$,
 \item simple objects in $\P(\phi)$ are called \emph{stable}, and
 \item define $\phi_\sigma^+(E):=\phi_1$ and $\phi_\sigma^-(E):=\phi_n$.
\end{itemize}

A stability condition $(Z,\P)$ is called {\em locally finite} if there
exists $\varepsilon>0$ such that $\P(\phi-\varepsilon,\phi+\varepsilon)$
is of finite length for each $\phi$, that is, the category is Artinian
and Noetherian.
Let $\stab(\T)$ be the set of all locally-finite stability conditions on
$\T$.
Bridgeland defined a generalized metric on $\stab(\T)$ as
\[
 d(\sigma_1,\sigma_2)=\sup_{0\neq E\in\T}\left\{
 |\phi_{\sigma_2}^-(E)-\phi_{\sigma_1}^-(E)|,
 |\phi_{\sigma_2}^+(E)-\phi_{\sigma_1}^+(E)|,
 \left|\log \frac{m_{\sigma_2}(E)}{m_{\sigma_1}(E)}\right|
 \right\}\in[0,\infty]
\]
for $\sigma_1,\sigma_2\in\stab(\T)$.
\begin{thm}
 \cite[Theorem 1.2]{bri}
 For each connected component $\Sigma\subset\stab(\T)$ there are a
 linear subspace $V(\Sigma)\subset\hom_\Z(K(\T),\C)$, with a
 well-defined linear topology, and a local homeomorphism
 $\z:\Sigma\rightarrow V(\Sigma)$ which maps a stability condition
 $(Z,\P)$ to its central charge $Z$.
\end{thm}

According to Bridgeland, there is a relation between a stability
condition and a heart of a bounded t-structure on $\T$.
Before recalling it we make some definitions.
\begin{df}
 \begin{enumerate}
  \item A \emph{t-structure} on $\T$ is a full subcategory $\mathcal{F}$,
	closed under isomorphism, satisfying;
	\begin{itemize}
	 \item $\mathcal{F}[1]\subset \mathcal{F}$ and
	 \item for every object $X\in\T$, there is a triangle
	       $A\rightarrow X \rightarrow B \rightarrow A[1]$
	       such that $A\in\mathcal{F}$ and $B\in\mathcal{F}^{\perp}$
	       where $\mathcal{F}^\perp=\{\, B\in\T \,|\, \hom_\T(A,B)=0%
\text{ for all } A\in\mathcal{F} \,\}$.
	\end{itemize}
  \item The \emph{heart} of a t-structure $\mathcal{F}\subset\T$ is the
	full subcategory $\A=\mathcal{F}\cap\mathcal{F}^\perp[1]$.
  \item A t-structure $\mathcal{F}\subset\T$ is said to be
	\emph{bounded} if
	$\T=\bigcup_{i,j\in\Z}\mathcal{F}[i]\cap\mathcal{F}^\perp[j]$.
 \end{enumerate}
\end{df}
\begin{df}
 Let $\A \subset \T$ be a full abelian subcategory.
 \begin{enumerate}
  \item The Grothendieck group $K(\A)$ of $\A$ is the quotient group of
	the free abelian group generated by all isomorphism classes in
	$\A$ over the subgroup generated by the elements of the form
	$[A]+[B]-[C]$ for each short exact sequence
	$0\rightarrow A\rightarrow C\rightarrow B\rightarrow 0$ in $\A$.
  \item A \emph{stability function} on $\A$ is a group homomorphism
	$Z:K(\A)\rightarrow\C$ satisfying $Z([E])=m(E)\,e^{i\pi\phi(E)}$
	where $m(E)>0$ and $0<\phi(E)\leq 1$ for each nonzero object
	$E\in\A$. We call $\phi(E)$ the \emph{phase} of $E$.
  \item An object $E\in\A$ is said to be \emph{semistable} if
	$\phi(F)\leq\phi(E)$ for each subobject $F\subset E$.
  \item A stability function $Z$ on $\A$ is said to have the
	\emph{Harder-Narasimhan property} if for every nonzero object
	$E\in\A$, there is a finite chain of subobjects
	\[
	 0=E_0\subset E_1 \subset\dots\subset E_{n-1} \subset E_n=E
	\]
	whose factors $F_j=E_j/E_{j-1}$ are semistable objects with
	\[
	 \phi(F_1)>\phi(F_2)>\dots>\phi(F_n).
	\]
 \end{enumerate}
\end{df}
\begin{thm}
 \cite[Proposition 5.3]{bri} \label{briprop53}
 To give a stability condition on a triangulated category $\T$ is
 equivalent to giving a bounded t-structure on $\T$ and a stability
 function on its heart with the Harder-Narasimhan property.
\end{thm}

\subsection{Actions on the space of stability conditions}
\label{subsectionact}

Bridgeland defined a right action of the group
$\widetilde{\mathrm{GL}^+(2,\mathbb{R})}$, the universal cover of
$\mathrm{GL}^+(2,\mathbb{R})$, and a left action of $\aut(\T)$, the
group of exact autoequivalences of $\T$, on $\stab(\T)$.
These actions commute with each other.

Note that $\widetilde{\mathrm{GL}^+(2,\mathbb{R})}$ can be thought of as
the set of pairs $(T,f)$, where $f:\mathbb{R}\rightarrow\mathbb{R}$
is an increasing map with $f(\phi+1)=f(\phi)+1$ and
$T\in\mathrm{GL}^+(2,\mathbb{R})$, such that
$T(e^{2i\pi\phi})/|T(e^{2i\pi\phi})|=e^{2i\pi f(\phi)}$
where we identify $\C$ and $\R^2$.
For $(Z,\P)\in\stab(\T)$, $(Z',\P')=(Z,\P)\cdot(T,f)$ is defined as
$Z'=T^{-1}\circ Z$ and $\P'(\phi)=\P(f(\phi))$.

Next, note that an element $G\in\aut(\T)$ induces an automorphism $g$
of $K(\T)$.
For $\sigma=(Z,\P)$, $G\cdot\sigma$ is defined as
$(Z\circ g^{-1},\P'')$ where $\P''(\phi)=G(\P(\phi))$.

The action of $\C$ on $\stab(\T)$, which is a part of
$\widetilde{\mathrm{GL}^+(2,\R)}$-action, is described in
Okada~\cite[Definition 2.3]{oka};
\begin{center}
 $z\cdot (Z,\P)=(Z',\P')$ where
 $Z'(E)=e^z\,Z(E)$ and $\P'(\phi)=\P(\phi-\frac{b}{\pi})$
\end{center}
for $z=a+bi\in\C$ and $(Z,\P)\in\stab(\T)$.

\subsection{Stability conditions and exceptional objects}
\label{subsection22}

Macr{\`{\i}} discussed the spaces of stability conditions on
a triangulated category $\T$ generated by finitely many exceptional
objects in \cite{mac}.
\begin{df}
 \begin{enumerate}
  \item An object $E$ in $\T$ is called	\emph{exceptional} if
	\[
 	 \hom^k_\T(E,E) =
	 \begin{cases}
	  \C & (k=0) \\
	  0 & (\text{otherwise})
	 \end{cases}.
	\]
  \item A sequence $(E_1,E_2,\dots,E_n)$ is called
	an \emph{exceptional collection} if each $E_i$ is exceptional
	and $\hom^k_\T(E_i,E_j)=0$ for all $k$ and $i>j$.
	In particular, $(E_1,E_2)$, consisting of two exceptional
	objects, is called an \emph{exceptional pair}.
  \item An exceptional collection is called \emph{complete} if $\{E_i\}$
	generates $\T$ by shifts and extensions.
  \item An exceptional collection is called \emph{Ext-exceptional}
	if $\hom^{\leq 0}(E_i,E_j)=0$ for all $i\neq j$.
 \end{enumerate}
\end{df}

\begin{df}
 \begin{enumerate}
  \item Let $E$ and $F$ be exceptional objects.
	We define $\mathcal{L}_EF$ and $\mathcal{R}_FE$	by the following
	distinguished triangles:
	\[
	 \mathcal{L}_EF \longrightarrow \hom^\bullet(E,F) \otimes E
	 \longrightarrow F,
	\]
	\[
	 E \longrightarrow \hom^\bullet(E,F)^* \otimes F
	 \longrightarrow \mathcal{R}_FE.
	\]
	We call $\mathcal{L}_EF$ \emph{the left mutation} of $F$ by $E$
	and $\mathcal{R}_FE$ \emph{the right mutation} of $E$ by $F$.
  \item Let $\mathcal{E}=(E_1,\dots,E_n)$ be an exceptional collection.
	We define $\mathcal{L}_i\mathcal{E}$ and
	$\mathcal{R}_i\mathcal{E}$ as the sequences
	\begin{align*}
	 \mathcal{L}_i\mathcal{E}&=
	 (E_1,\dots,E_{i-1},\mathcal{L}_{E_i}E_{i+1},E_i,
	  E_{i+2},\dots,E_n), \\
	 \mathcal{R}_i\mathcal{E}&=
	 (E_1,\dots,E_{i-1},E_{i+1},\mathcal{R}_{E_{i+1}}E_{i},
	  E_{i+2},\dots,E_n).
	\end{align*}
	We call $\mathcal{L}_i\mathcal{E}$
	(resp. $\mathcal{R}_i\mathcal{E}$)
	\emph{the left (resp. right) mutation} of $\mathcal{E}$.
 \end{enumerate}
\end{df}
\begin{prop}
 (See for example \cite{bon}.)
 \begin{enumerate}
  \item A mutation of an exceptional collection is also an exceptional
	collection.
  \item If an exceptional collection generates $\T$ then its mutation
	also generates $\T$.
  \item \label{mutbra}
	The following relations hold:
	\begin{align*}
	 \mathcal{L}_i\mathcal{R}_i=\mathcal{R}_i\mathcal{L}_i=\text{id},
	 \quad & \mathcal{R}_i\mathcal{R}_{i+1}\mathcal{R}_i
	 =\mathcal{R}_{i+1}\mathcal{R}_i\mathcal{R}_{i+1}, \\
	 &\quad\text{and}\quad
	 \mathcal{L}_i\mathcal{L}_{i+1}\mathcal{L}_i
	 =\mathcal{L}_{i+1}\mathcal{L}_i\mathcal{L}_{i+1}. 
	\end{align*}
 \end{enumerate}
\end{prop}

For an exceptional collection $(E_1,\dots,E_n)$, we denote by
$\langle E_1,\dots,E_n \rangle$ the smallest extension-closed full
subcategory of $\T$ containing $\{E_1,\dots,E_n\}$.

\begin{lem}
 \cite[Lemma 3.14]{mac} \label{mac314}
 Let $(E_1,\dots,E_n)$ be a complete Ext-exceptional collection on $\T$.
 Then $\langle E_1,\dots,E_n \rangle$ is the heart of a bounded
 t-structure on $\T$.
\end{lem}

\begin{lem}
 \cite[Lemma 3.16]{mac} \label{maclem316}
 Let $(E_1,\dots,E_n)$ be a complete Ext-exceptional collection on $\T$
 and let $(Z,\P)$ be a stability condition on $\T$.
 Assume $E_1,\dots,E_n\in\P((0,1])$.
 Then $\langle E_1,\dots,E_n \rangle=\P((0,1])$ and $E_j$ is stable,
 for all $j=1,\dots,n$.
\end{lem}

Given a complete exceptional collection $(E_1,\dots,E_n)$ on $\T$,
the Grothendieck group $K(\T)$ is a free abelian group of finite rank
isomorphic to $\Z^{\oplus n}$ generated by the isomorphism classes
$[E_i]$.
Macr{\`{\i}} showed how to construct stability conditions from an
exceptional collection.

\begin{lem}
 \cite{mac} \label{maclemma}
 Let $\E=(E_1,\dots,E_n)$ be a complete exceptional collection on $\T$.
 A stability condition, as a pair $(Z,\mathcal{A})$ of a heart
 $\mathcal{A}$ of a bounded t-structure and a stability function $Z$ on
 $\mathcal{A}$, is constructed by the following steps:
 \begin{itemize}
  \item Choose integers $p_1,\dots,p_n$ such that
	$(E_1[p_1],\dots,E_n[p_n])$ is Ext-exceptional.
  \item Denote $\mathrm{Q}_p:=\langle E_1[p_1],\dots,E_n[p_n] \rangle$,
	which is the heart of a bounded	t-structure.
	(Lemma~\ref{mac314}.)
  \item Pick $n$ points $z_1,\dots,z_n$ in
	$H=\{\,re^{i\pi\phi}\in\C\,|\,r>0,\,0<\phi\leq 1\}$ and define a
	homomorphism  $Z_p:K(Q_p) \rightarrow \C$ as $Z_p(E_i[p_i])=z_i$.
	Then $Z_p$ is a stability function with the Harder-Narasimhan
	property.
  \item The pair $(Z_p,\mathrm{Q}_p)$ is the stability condition.
	(Theorem \ref{briprop53}.)
 \end{itemize}
\end{lem}

Let $\Theta_\E$ be the subset of $\stab(\T)$ consisting of all stability
conditions obtained by Lemma \ref{maclemma}, up to the action of
$\widetilde{\mathrm{GL}^+(2,\mathbb{R})}$.
Lemma \ref{maclem316} implies the following immediately.

\begin{lem}
 \cite{mac}
 Let $\E=(E_1,\dots,E_n)$ be a complete exceptional collection.
 Then $E_i$'s are stable in each stability condition of $\Theta_\E$.
\end{lem}

In general, $\Theta_\E$ is not the subspace consisting of
stability conditions in which $E_i$'s are stable.

\begin{lem}
 \label{mac319} \cite[Lemma 3.19.]{mac}
 The subspace $\Theta_\E\subset\stab(\T)$ is an open, connected and
 simply connected $n$-dimensional submanifold.
\end{lem}

It is worth noting how Macr{\`{\i}} proved Lemma~\ref{mac319}.

For an exceptional collection $\mathcal{F}_s=(F_1,\dots,F_s)$ ($s>1$),
we define, for $i<j$,
\[
 k_{i,j}^{\mathcal{F}_s}:=
 \begin{cases}
  +\infty, & \text{if } \hom^k(F_i,F_j)=0 \text{ for all k}, \\
  \min\left\{k\,\left|\, \hom^k(F_i,F_j) \neq 0 \right.\right\},
  & \text{otherwise.}
 \end{cases}
\]
Then define $\alpha_s^{\mathcal{F}_s}=0$, and for $i<s$,
\[
 \alpha_i^{\mathcal{F}_s} :=
 \min_{j>i}\left\{ k_{i,j}^{\mathcal{F}_s}+\alpha_j^{\mathcal{F}_s}
                   \right\}-(s-i-1)
\]
inductively.
Consider $\mathbb{R}^{n}$ with coordinates $\phi_1,\dots,\phi_n$.
Let
$\mathcal{F}_s := (E_{\ell_1},\dots,E_{\ell_s})%
\subset (E_1,\dots,E_n)$, $s>1$.
Define $R^{\mathcal{F}_s}$ on $\mathbb{R}^n$ as the relation
$\phi_{\ell_1}<\phi_{\ell_s}+\alpha_1^{\mathcal{F}_s}$.
Finally define
\[
 \cc_\E := \left\{
 (m_1,\dots,m_n,\phi_1,\dots,\phi_n) \in \mathbb{R}^{2n} \,\left|\,
 \begin{array}{l}
  m_i > 0 \text{ for all } i \text{ and }\\
  R^{\mathcal{F}_s} \text{ for all } \mathcal{F}_s \subset \E
   \text{, } s>1
 \end{array}
 \right.\right\}.
\]
Then $\cc_\E$ is homeomorphic to $\Theta_\E$ via the map
$\rho:\Theta_\E \rightarrow \mathcal{C}_\E$ defined by
$m_i(\rho(\sigma)):=|Z(E_i)|$ and
$\phi_i(\rho(\sigma)):=\phi_\sigma(E_i)$.
See \cite{mac} for the proof that $\rho$ is a homeomorphism and
that $\mathcal{C}_\E$ is an open, connected and simply connected
$n$-dimensional manifold.

\section{Quivers with two vertices}

A \emph{quiver} $Q$ consists of the following data;
a set $Q_0$ of vertices, a set $Q_1$ of arrows,
a map $s:Q_1\rightarrow Q_0$ called the source map, and
a map $t:Q_1\rightarrow Q_0$ called the target map.
We write $\alpha:a\rightarrow b$ if $s(\alpha)=a$ and $t(\alpha)=b$ for
$\alpha\in Q_1$.
A \emph{finite dimensional representation} $V$ on $Q$ is a family
$V=\{V_a,f_\alpha\}$ such that
$V_a$ is a finite dimensional $\C$-vector space for each $a\in Q_0$ and
$f_\alpha:V_a\rightarrow V_b$ is a linear map for each arrow
$\alpha:a\rightarrow b$.
A \emph{morphism} $F:V\rightarrow V'$ between representations is a
family of linear maps $F=\{F_a:V_a\rightarrow V'_a\}$ such that
$f'_\alpha\circ F_a=F_b\circ f_\alpha$ for each $\alpha:a\rightarrow b$.
We denote by $\rep(Q)$ the abelian category consisting of all finite
dimensional representations on $Q$ and all morphisms; and denote by
$\D(Q)=\D^b(\rep(Q))$ the bounded derived category of $\rep(Q)$.

In this article, we consider the quiver with finite vertices, finite
arrows, no loops, no oriented cycles, and no relations on arrows.
For such quiver, following results are well-known:
\begin{itemize}
 \item Let $E_a\in\rep(Q)$ be a representation assigning $\C$ to the
       vertex $a$ and $0$ to other vertices.
       Then $\{E_a\}_{a\in Q_0}$ is a complete set of simple objects in
       $\rep(Q)$.
 \item $\hom(E_a,E_b)=\C$ if $a=b$ and $=0$ if $a\neq b$.
 \item $\ext^1(E_a,E_b)=\C^n$ if there exist $n$ arrows from $a$ to $b$.
 \item Denote by $E_a\in\D(Q)$ the complex assigning $E_a$ to degree $0$
       and $0$ to other degree.
       Then $E_a$ is an exceptional object for each $a\in Q_0$.
 \item $K(\D(Q))\cong K(\rep(Q))$ is the free abelian group generated by
       $\{E_a\}_{a\in Q_0}$.
 \item $\rep(Q)$ is a hereditary abelian category.
       Hence each $X\in\D(Q)$ is isomorphic to its cohomology regarded
       as a complex with zero differential. (e.g. \cite{kra}.)
\end{itemize}

By ordering in a suitable manner, the sequence $(E_a)_{a_\in Q_0}$ becomes a
complete exceptional collection on $\D(Q)$.
Hence we apply results in Subsection~\ref{subsection22} to $\D(Q)$.
In fact, Macr{\`{\i}} studied $\stab(P_n)=\stab(\D(P_n))$ where $P_n$ is
the quiver as follows:
\[
 \xymatrix{
 P_n\,:\, {\stackrel{0}{\bullet}} \ar @<8pt> [r] \ar @<3pt> [r]
 \ar @<-8pt> [r] ^-{\scalebox{.7}{\vdots}}
 & {\stackrel{1}{\bullet}}
 }
\]

Let $S_0$ and $S_1$ be exceptional objects in $\D(P_n)$ such that
$S_0[1]$ and $S_1$ are simple objects of $\rep(P_n)$.
According to \cite{mac, boe, rin}, if we define
\[
 S_k :=
 \begin{cases}
  \mathcal{L}_{S_{k+1}}S_{k+2} & (k<0) \\
  \mathcal{R}_{S_{k-1}}S_{k-2} & (k\geq 2)
 \end{cases}
\]
inductively, then $\{S_k\}$ is the complete set of exceptional objects
in $\D(P_n)$, up to shifts and isomorphisms.
Each adjacent pair $(S_k,S_{k+1})$ is an exceptional pair.
Moreover each $(S_k,S_{k+1})$ is the right mutation of $(S_{k-1},S_k)$.
Note that, if $n=1$, there are only three exceptional objects up to
shifts and isomorphisms, i.e. $S_0$, $S_1$, and $S_2$, up to shifts and
isomorphisms.
Further, the right mutation satisfies:
$\mathcal{R}_1\mathcal{R}_1\mathcal{R}_1(S_0,S_1)=(S_0[1],S_1[1])$.

\begin{lem}
 \cite[Lemma 4.1]{mac}
 \label{mac41}
 Assume that $n>1$, if $i<j$ then in $\D(P_n)$
 \begin{itemize}
  \item $\hom^k(S_i,S_j)\neq 0$ only if $k=0$;
  \item $\hom^k(S_j,S_i)\neq 0$ only if $k=1$.
 \end{itemize}
 In particular the pair $(S_k,S_{k+1})$ is a complete exceptional pair.
\end{lem}

\begin{lem}
 \cite[Lemma 4.2]{mac}
 \label{mac42}
 In every stability condition on $\D(P_n)$ there exists a stable
 exceptional pair $(E,F)$.
\end{lem}

Let $\Theta_k$, $k\in\Z$, be a subset of $\stab(P_n)$ consisting of
all stability conditions made from $(S_k,S_{k+1})$ by Lemma
\ref{maclemma}.

\begin{lem}
 \label{mac4a}
 $\Theta_k$ coincides with the subset of $\stab(P_n)$ consisting of
 all stability conditions in which $S_k$ and $S_{k+1}$ are stable.
\end{lem}
\begin{proof}
 See \cite[Section~4]{mac}.
\end{proof}

\begin{prop}
 \label{macprop}
 If $n>1$, $\stab(P_n)=\bigcup_{k\in\Z}\Theta_k$.
 In addition, $\stab(P_1)=\Theta_0 \cup \Theta_1 \cup \Theta_2$.
\end{prop}
\begin{proof}
 It follows immediately from Lemma \ref{mac41}, \ref{mac42} and
 \ref{mac4a}.
\end{proof}

\begin{prop}
 \cite[Proposition 4.4]{mac}
 For all integers $k\neq h$ we have
 \[
  \Theta_k \cap \Theta_h = O_{-1}
 \]
 where $O_{-1}$ is the $\widetilde{\mathrm{GL}^+(2,\R)}$-orbit of the
 stability condition $\sigma_{-1}=(Z_{-1},\P_{-1})$ given by
 $Z_{-1}(S_0[1])=-1$ and $Z_{-1}(S_1)=1+i$.
\end{prop}

Applying the proof of Lemma \ref{mac319}, we have
\[
 \Theta_k\cong \cc_k :=
 \{(m_1,m_2,\phi_1,\phi_2)\in\mathbb{R}^4\,|\,
   m_i>0 \text{ and } \phi_1<\phi_2\}\,.
\]
Since it is connected and simply connected, $\stab(P_n)$ is connected
and simply connected by using of the Seifert-van Kampen theorem.

\begin{thm}
 \cite[Theorem 4.5]{mac}
 $\stab(P_n)$ is a connected and simply connected $2$-dimensional
 complex manifold.
\end{thm}

To analyze how $\Theta_k$'s are glued to each other, we change coordinates
of $\cc_k$.

\begin{lem}
 \label{ctimesh}
 $\cc_k\cong\C\times\H$ where $\H=\{z\in\C \,|\, \im z>0\}$,
 the upper half plane.
\end{lem}
\begin{proof}
 Define
 \[
  \zeta_k(m_1,m_2,\phi_1,\phi_2):=\left(\log(m_1)+i\pi\phi_1,\,%
\log\left(\frac{m_2}{m_1}\right)+i\pi(\phi_2-\phi_1)\right)\,.
 \]
 It is clear that $\zeta_k:\cc_k\rightarrow\C\times\H$ is a
 homeomorphism.
\end{proof}

Let us note the correspondence between elements in $\Theta_k$ and
$\C\times\H$.
For $z=a+bi\in\C$ and $w=c+di\in\H$, the corresponding stability
condition $\sigma=(Z,\P)\in\Theta_k$ is given by
\begin{equation}
 \label{equationofk}
 Z(S_k)=e^z,\,\,Z(S_{k+1})=e^{z+w},\,\,
 \phi_\sigma(S_k)=\frac{b}{\pi}\,\,\text{ and }\,\,
 \phi_\sigma(S_{k+1})=\frac{b+d}{\pi}.
\end{equation}
Conversely, for $\sigma=(Z,\P)\in\Theta_k$,
\begin{equation*}
  \begin{split}
   z&=\log\left|Z(S_k)\right|+i\pi\phi_\sigma(S_k)\,\text{ and } \\
   w&=\log\left|\frac{Z(S_{k+1})}{Z(S_k)}\right|
      +i\pi(\phi_\sigma(S_{k+1})-\phi_\sigma(S_k)).
  \end{split}
\end{equation*}

\begin{lem}
 The action of $\C$ on $\stab(P_n)$, described by Okada \cite{oka}
 (see Subsection~\ref{subsectionact}), preserves $\Theta_k$ for all $k$.
\end{lem}
\begin{proof}
 The $\C$-action preserves stable objects.
\end{proof}

The $\C$-action on $\C\times\H$ via homeomorphism
$\Theta_k\cong\C\times\H$ is given by
\begin{equation}
 \label{actionofc}
 z'\cdot(z,w) = (z'+z,w).
\end{equation}
Therefore the orbit space is homeomorphic to $\H$.

\section{The case of one arrow}

In this section, we analyze $\stab(P_1)$.
Section~3 says
\[
 \stab(P_1)=\Theta_0\cup\Theta_1\cup\Theta_2
\]
where $\Theta_k$ are related to $(S_0,S_1)$, $(S_1,S_2)$, and
$(S_2,S_0[1])$, respectively.
The intersection $O_{-1}=\Theta_k\cap\Theta_h$ (for each $k\neq h$)
consists of all stability conditions in which $S_0$, $S_1$ and $S_2$ are
stable.
For each $\sigma=(Z,\P)\in O_{-1}$, since there is a distinguished
triangle
\[
 S_0 \longrightarrow S_1 \longrightarrow S_2
 \longrightarrow S_0[1],
\]
axioms of the stability condition induce
\begin{equation}
 \label{p11}
 Z(S_1)=Z(S_0)+Z(S_2)
\end{equation}
and
\begin{equation}
 \label{p12}
 \phi_\sigma(S_0)<\phi_\sigma(S_1)<\phi_\sigma(S_2)<
 \phi_\sigma(S_0)+1.
\end{equation}

Put $\Theta_k\cong\C_k\times\H_k$ ($k=0,1,2$) where $\C_k$ and $\H_k$
are copies of $\C$ and $\H$ (Lemma~\ref{ctimesh}).
Now we analyze how $\C_k\times\H_k$ are glued each other.

\begin{thm}
 \label{stabp1}
 Put
 \[
  \overline{\H}_k:=\{\,w\in\H\,|\, 0<\im w<\pi\,\}\ (k=0,1,2).
 \]
 Define
 $\varphi_k:\C_0\times\overline{\H}_0\rightarrow\C_k\times\overline{\H}_k$
 ($k=1,2$) as
 \begin{align*}
  \varphi_1(z,w) &:=
   \left(z+w,\,\log\frac{e^w-1}{e^w}\right) \quad\text{and} \\
  \varphi_2(z,w) &:=
   \left(z+\log(e^w-1),\,\log\frac{1}{1-e^w}\right)\,.
 \end{align*}
 Then
 \[
  \stab(P_1) \cong \left(\bigcup_{k=0}^{2}\C_k\times\H_k\right)
  \big/\sim
 \]
 where the equivalence relation is generated by
 $(z,w)\sim\varphi_k(z,w)$ for each $(z,w)\in\C_0\times\overline{\H}_0$
 and $k$.
\end{thm}
\begin{proof}
 Because of the inequality (\ref{p12}), the intersection
 $O_{-1}\subset\Theta_k$ corresponds to $\C_k\times\overline{\H}_k$ for
 each $k$.
 Let $(Z,\P)\in O_{-1}$ and $(z_k,w_k)\in\C_k\times\overline{\H}_k$
 ($k=0$, $1$, $2$) be elements corresponding to each other.
 The equations \eqref{equationofk} in Section~3 imply that
 \[
  e^{z_1}=Z(S_1)=e^{z_0+w_0} \quad\text{ and }\quad
  e^{z_1+w_1}=Z(S_2)=e^{z_0+w_0}-e^{z_0}
 \]
 and
 \[
  e^{z_2}=Z(S_2)=e^{z_0+w_0}-e^{z_0} \quad\text{ and }\quad
  e^{z_2+w_2}=Z(S_0[1])=-e^{z_0}
 \]
 (see Figure~\ref{zsiofp1}).
 Hence $\varphi_k$ are defined.
 It is clear that $\varphi_k$ are homeomorphisms.
\end{proof}

\begin{figure}[ht]
\begin{center}
 {\footnotesize
 \begin{picture}(80,70)(-40,-20)
  \put(0,-20){\line(0,1){60}}
  \put(-40,0){\line(1,0){80}}
  \thicklines
  \put(0,0){\vector(2,1){30}}
  \put(0,0){\vector(1,3){10}}
  \thinlines
  \dottedline{2}(0,0)(-20,15) \put(-20,15){\vector(-4,3){0}}
  \dottedline{2}(0,0)(-30,-15) \put(-30,-15){\vector(-2,-1){0}}
  \put(20,17){\makebox(0,0)[lb]{$Z(S_0)=e^{z_0}$}}
  \put(5,32){\makebox(0,0)[lb]{$Z(S_1)=e^{z_0+w_0}$}}
  \put(-10,17){\makebox(0,0)[rb]{$Z(S_2)$}}
  \put(-32,-15){\makebox(0,0)[rt]{$Z(S_0[1])$}}
  \put(0,42){\makebox(0,0)[b]{$\Theta_0$}}
 \end{picture}
 \qquad\quad
 \begin{picture}(80,70)(-40,-20)
  \put(0,-20){\line(0,1){60}}
  \put(-40,0){\line(1,0){80}}
  \thicklines
  \put(0,0){\vector(1,3){10}}
  \put(0,0){\vector(-4,3){20}}
  \thinlines
  \dottedline{2}(0,0)(30,15) \put(30,15){\vector(2,1){0}}
  \dottedline{2}(0,0)(-30,-15) \put(-30,-15){\vector(-2,-1){0}}
  \put(10,32){\makebox(0,0)[lb]{$e^{z_1}$}}
  \put(-15,17){\makebox(0,0)[b]{$e^{z_1+w_1}$}}
  \put(0,42){\makebox(0,0)[b]{$\Theta_1$}}
 \end{picture}
 \qquad\quad
 \begin{picture}(80,70)(-40,-20)
  \put(0,-20){\line(0,1){60}}
  \put(-40,0){\line(1,0){80}}
  \thicklines
  \put(0,0){\vector(-4,3){20}}
  \put(0,0){\vector(-2,-1){30}}
  \thinlines
  \dottedline{2}(0,0)(30,15) \put(30,15){\vector(2,1){0}}
  \dottedline{2}(0,0)(10,30) \put(10,30){\vector(1,3){0}}
  \put(-20,17){\makebox(0,0)[rb]{$e^{z_2}$}}
  \put(-30,-15){\makebox(0,0)[t]{$e^{z_2+w_2}$}}
  \put(0,42){\makebox(0,0)[b]{$\Theta_2$}}
 \end{picture}
 }
\end{center}
\caption{} \label{zsiofp1}
\end{figure}

\begin{thm}
 \label{thmofccn}
 Define
 $\psi_k:\overline{\H}_0\rightarrow\overline{\H}_k$ ($k=1,2$)
 as
 \[
  \psi_1(w) := \log\frac{e^w-1}{e^w}
  \quad\text{ and }\quad
  \psi_2(w) := \log\frac{1}{1-e^w}
 \]
 and define
 \[
  \cc_1:=\left(\bigcup_{k=0}^{2}\H_k\right)\big/\sim
 \]
 where the equivalence relation is generated by $w\sim\psi_k(w)$ for
 each $w\in\overline{\H}_0$ and $k$.
 Then $\stab(P_1)/\C \cong \cc_1$.
 (See Figure~\ref{c1} and Corollary~\ref{maincor}.)
\end{thm}
\begin{proof}
 The proof is immediately from Theorem \ref{stabp1} and the $\C$-action
 on $\C_k\times\H_k$, described in \eqref{actionofc}, Section 3.
\end{proof}

\section{The case of multiple arrows}

In this section, we analyze $\stab(P_n)$ for $n\geq 2$.
Section~3 says that, in $\D(P_n)$, there is a complete set of
exceptional objects
\[
 \{\,\dots, S_{-2},\, S_{-1},\, S_0,\, S_1,\, S_2,\,\dots\,\}
\]
where each adjacent pair $(S_k,S_{k+1})$ is the complete strong
exceptional pair and is the right mutation of $(S_{k-1},S_k)$.

\begin{df}
 \label{defofak}
 Let $n\geq 2$.
 Define a sequence of real numbers $\{a_k\}_{k\in\Z}$ as
 \begin{itemize}
  \item $a_0=0$, $a_1=1$, $a_k=na_{k-1}-a_{k-2}$ for $k\geq 2$, and
  \item $a_{-\ell}=-a_{\ell}$ for $\ell\geq 1$.
 \end{itemize}
\end{df}

\begin{lem}
 \label{eseqofpn}
 In $\D(P_n)$, there is a distinguished triangle
 \[
  S_{k-2} \longrightarrow S_{k-1}^{\oplus n} \longrightarrow S_k
 \longrightarrow S_{k-2}[1].
 \]
 Moreover, for a group homomorphism $Z:K(P_n)\rightarrow \C$,
 the following equation hold:
 \begin{equation}
  \label{pneq1}
  Z(S_k)=a_kZ(S_1)-a_{k-1}Z(S_0)
 \end{equation}
\end{lem}
\begin{proof}
 It is well-known that
 \[
  \hom_{\D^b(P_n)}^j(S_0,S_1) \cong
  \begin{cases}
   \C^n & (j=0) \\
   0 & (\text{otherwise})
  \end{cases}.
 \]
 By the definition of the right mutation, there is a triangle
 \[
  S_0 \longrightarrow \hom_{D^b(P_n)}^\bullet(S_0,S_1)^*\otimes S_1
  =S_1^{\oplus n} \longrightarrow \mathcal{R}_{S_1}S_0=S_2.
 \]
 By applying $\hom_{\D(P_n)}^j(S_1,\quad)$, we obtain
 \[
  \hom_{\D(P_n)}^j(S_1,S_2) \cong
  \begin{cases}
   \C^n & (j=0) \\
   0 & (\text{otherwise})
  \end{cases}.
 \]

 Assume that the following condition hold:
 \[
  \hom_{\D(P_n)}^j(S_{k-2},S_{k-1})\cong
  \begin{cases}
   \C^n & (j=0) \\
   0 & (\text{otherwise})
  \end{cases}
 \]
 The definition of the right mutation induces the triangle
 \[
  S_{k-2} \longrightarrow
  \hom_{D^b(P_n)}^\bullet(S_{k-2},S_{k-1})^*\otimes S_{k-1}
  =S_{k-1}^{\oplus n} \longrightarrow
  \mathcal{R}_{S_{k-1}}S_{k-2}=S_k.
 \]
 By applying $\hom_{\D(P_n)}^j(S_{k-1},\quad)$, we obtain
 \[
  \hom_{\D(P_n)}^j(S_{k-1},S_k)\cong
  \begin{cases}
   \C^n & (j=0) \\
   0 & (\text{otherwise})
  \end{cases}.
 \]
 By induction, we obtain the triangle, which we want, for $k\geq 2$.
 For a group homomorphism $Z:K(P_n)\rightarrow\C$, the triangle induces
 \[
  Z(S_k)=nZ(S_{k-1})-Z(S_{k-2}).
 \]
 By induction, we obtain the equation~(\ref{pneq1}) for $k\geq 2$.

 To prove Lemma for $k<0$, it is enough to do similar, using the
 definition of the left mutation.
\end{proof}

The sequence $\{a_k\}_{k\geq 0}$ is described as follows;
if $n=2$,
\[
 a_0=0,\,\,a_1=1,\,\,a_2=2,\,\,\dots,\,\,a_k=k,\dots\,,
\]
and if $n>2$,
\[
 \begin{split}
  &a_0=0,\,\,a_1=1,\,\,a_2=n,\,\,a_3=n^2-1,\dots, \\
  &a_k=\frac{1}{\sqrt{n^2-4}}\left\{
  \left(\frac{n+\sqrt{n^2-4}}{2}\right)^k-
  \left(\frac{n-\sqrt{n^2-4}}{2}\right)^k\right\}
  ,\dots
 \end{split}
\]
The fractional sequences $\left\{\frac{a_{k+1}}{a_k}\right\}$ and
$\left\{\frac{a_k}{a_{k+1}}\right\}$ are convergent sequences:
\begin{equation}
\label{ak}
 \begin{split}
  &
  \frac{a_2}{a_1}>\frac{a_3}{a_2}>\dots>\frac{a_k}{a_{k-1}}
  >\frac{a_{k+1}}{a_k}>\cdots\xrightarrow{k\to\infty}
  \frac{n+\sqrt{n^2-4}}{2} \\
  &
  \frac{a_1}{a_2}<\frac{a_2}{a_3}<\dots<\frac{a_{k-1}}{a_k}
  <\frac{a_k}{a_{k+1}}<\cdots\xrightarrow{k\to\infty}
  \frac{n-\sqrt{n^2-4}}{2}
 \end{split}
\end{equation}

We now ready to analyze how $\Theta_k$'s are glued each other.
As we saw in Section~3, the space of stability conditions is a union of
$\Theta_k$'s:
\[
 \stab(P_n)\cong\bigcup_{k\in\Z}\Theta_k\,,
\]
where each $\Theta_k$ is the subset consisting of all stability
conditions in which $S_k$ and $S_{k+1}$ are stable.
$\Theta_k$ is homeomorphic to $\C_k\times\H_k$ where $\C_k$ and $\H_k$
are copies of $\C$ and $\H$ respectively.
The intersection $O_{-1}=\Theta_k\cap\Theta_h$ (for all $k\neq h$) is
the subset consisting of all stability conditions in which $S_k$ are
stable for all $k$.

\begin{thm}
 \label{stabpn}
 Define
 $\varphi_k:\C_0\times\overline{\H}_0\rightarrow\C_k\times\overline{\H}_k$
 as
 \begin{equation}
  \varphi_k(z,w):=\left(
   z+\log(a_ke^w-a_{k-1}),\,
   \log\frac{a_{k+1}e^w-a_k}{a_ke^w-a_{k-1}} \right)\,
 \end{equation}
 for each $k$.
 Then
 \[
  \stab(P_n)\cong
  \left(\bigcup_{k\in\Z}\C_k\times\H_k\right)\big/\sim
 \]
 where the equivalence relation is generated by
 $(z,w)\sim\varphi_k(z,w)$ for all
 $(z,w)\in\C_0\times\overline{\H}_0$ and $k$.
 Here, $\overline{\H}_k=\{\,w\in\H\,|\,0<\im w<1\,\}$.
\end{thm}
\begin{proof}
 It is clear that $O_{-1}$ corresponds to $\C_k\times\overline{\H}_k$
 for each $k$ (c.f. Lemma~\ref{eseqofpn}).
 Let $(z,w)\in\C_0\times\overline{\H}_0$ and
 $(z',w')\in\C_k\times\overline{\H}_k$ be elements corresponding to
 $(Z,\P)\in O_{-1}$.
 From (\ref{equationofk}) in Section~3, we have $Z(S_0)=e^z$,
 $Z(S_1)=e^{z+w}$, $Z(S_k)=e^{z'}$, and $Z(S_{k+1})=e^{z'+w'}$.
 These equations and \eqref{pneq1} in Lemma~\ref{eseqofpn} imply
 \[
  e^{z'}=a_ke^{z+w}-a_{k-1}e^z \quad\text{and}\quad
  e^{w'}=\frac{a_{k+1}e^{z+w}-a_ke^z}{a_ke^{z+w}-a_{k-1}e^z} \,.
 \]
 Hence $\varphi_k$ are defined.
 It is clear that $\varphi_k$ are homeomorphisms.
\end{proof}

\begin{thm}
 \label{defofpsik}
 Define
 $\psi_k:\overline{\H}_0\rightarrow\overline{\H}_k$
 as
 \[
  \psi_k(w):=\log\frac{a_{k+1}e^w-a_k}{a_ke^w-a_{k-1}}
 \]
 for each $k$ and define
 \[
  \cc_n:=\left(\bigcup_{k\in\Z}\H_k\right)\big/\sim
 \]
 where the equivalence relation is generated by
 $w\sim\psi_k(w)$ for each $w\in\overline{\H}_0$ and $k$.
 Then $\stab(P_n)/\C \cong \cc_n$.
 (See Figure~\ref{cn} and Corollary~\ref{maincor}.)
\end{thm}
\begin{proof}
 The proof is immediately from Theorem~\ref{stabpn} and the $\C$-action
 on $\C_k\times\H_k$, described in \eqref{actionofc}, Section~3.
\end{proof}

\section{Proof of main theorems}

In this section we study the covering map property of the local
homeomorphism $\z:\stab(P_n)\rightarrow\hom_\Z(K(P_n),\C)$.
Note that, as we mentioned in Section~3, the Grothendieck group $K(P_n)$
is a free abelian group generated by the isomorphism classes of simple
objects.
Thus we have $K(P_n)\cong\Z\langle [S_0],[S_1] \rangle$.
It is clear that $\hom_\Z(K(P_n),\C)\cong\C^2$ by mapping $Z$ to
$(Z(S_0),Z(S_1))$.

\begin{lem}
 The image of $\z:\stab(P_n)\rightarrow\hom_\Z(K(P_n),\C)$
 is homeomorphic to $\C^2\setminus\{(0,0)\}$.
\end{lem}
\begin{proof}
 It is enough to define a stability function $Z$ on
 $\langle S_k[p_k],S_{k+1}[p_{k+1}]\rangle$, a heart generated by an
 Ext-exceptional pair, such that $Z(S_i)=z_i$ ($i=0,1$) for each
 $(z_0,z_1)\in\C^2\setminus\{(0,0)\}$.

 Suppose $z_0\neq 0$ and $z_1\neq 0$.
 We define $Z$ on $\langle S_0[p_0],S_1[p_1]\rangle$ as follows:
 We choose $\epsilon_i\in \{0,1\}$ ($i=0,1$) such that
 $(-1)^{\epsilon_i}z_i$ are both in
 $H=\{\,re^{i\pi\phi}\,|\,r>0,\,0<\phi\leq 1\,\}$.
 We define $p_0=2-\epsilon_0$, $p_1=-\epsilon_1$, and
 $Z(S_i[p_i])=(-1)^{\epsilon_i}z_i$.
 It is easy to see that $(S_0[p_0],S_1[p_1])$ is an Ext-exceptional pair
 and $Z(S_i)=z_i$ ($i=0,1$).

 Suppose $z_0=0$ and $z_1\neq 0$.
 We define $Z$ on $\langle S_1[p_1],S_2[p_2] \rangle$ as follows:
 Choose an $\epsilon\in\{0,1\}$ such that $(-1)^{\epsilon}z_1$ is in
 $H$.
 We define $p_1=2-\epsilon$ and $p_2=-\epsilon$, and define
 $Z(S_1[p_1])=(-1)^{\epsilon}z_1$ and $Z(S_2[p_2])=(-1)^{\epsilon}nz_1$.

 Suppose $z_0\neq 0$ and $z_1=0$.
 We can define $Z$ on $\langle S_{-1}[p_{-1}],S_0[p_0]\rangle$ as
 follows:
 Choose an $\epsilon\in\{0,1\}$ such that $(-1)^{\epsilon}z_0$ is in
 $H$.
 We define $p_{-1}=2-\epsilon$ and $p_0=-\epsilon$, and define
 $Z(S_{-1}[p_{-1}])=(-1)^{\epsilon}nz_0$ and
 $Z(S_0[p_0])=(-1)^{\epsilon}z_0$.

 Since there are triangles
 $S_0\rightarrow S_1^{\oplus n}\rightarrow S_2$ and
 $S_{-1}\rightarrow S_0^{\oplus n}\rightarrow S_1$,
 it is also easy to see that $Z(S_0)=z_0$ and $Z(S_1)=z_1$.
\end{proof}

Let $g_k$ be a map defined in the diagram:
\[
 \xymatrix{
  \Theta_k \ar[r]^-{f}_-{\cong} \ar[d]_{\z}
  & \C_k\times\H_k \ar@{..>}[d]_{g_k} \\
  \hom_\Z(K(P_n),\C) \ar[r]^-{h}_-{\cong} & \C^2
 }
\]
The following equalities are easy to see:
\begin{align*}
 g_k(z,w)&=\left(
  \frac{a_ke^{z+w}-a_{k+1}e^z}{a_{k-1}a_{k+1}-a_k^2},
  \frac{a_{k-1}e^{z+w}-a_ke^z}{a_{k-1}a_{k+1}-a_k^2}
 \right) \\
 g_k(z'+z,w)&=e^{z'}\cdot g_k(z,w)
 \quad (\text{$g_k$ commutes with $\C$-actions.}) \\
 g_k\circ\varphi_k(z,w)&=(e^z,e^{z+w})
\end{align*}
Therefore the diagram
\[
 \xymatrix{
  \C \ar[d]_{\exp} \ar[r]
  & \bigcup_{k\in\Z}\C_k\times\H_k/_\sim \ar[d]_{g} \ar[r]
  & \bigcup_{k\in\Z}\H_k/_\sim \ar[d]_{\chi_n} \\
  \C^\times \ar[r]
  & \C^2\setminus\{(0,0)\} \ar[r]^-{\pi}
  & \CP^1
 }
\]
is commutative and both rows are principal fiber bundles.
This shows the theorem:
\begin{thm}
 We obtain the commutative diagram
 \[
  \xymatrix{
   \C \ar[d]_{\exp} \ar[r] & \stab(P_n) \ar[d]_{\z} \ar[r]
   & \cc_n \ar[d]_{\chi_n} \\
   \C^\times \ar[r] & \C^2\setminus\{(0,0)\} \ar[r]^-{\pi}
   & \CP^1
  }
 \]
 where both rows are principal fiber bundles.
\end{thm}

Before proving main theorem, we describe $\chi_n$.
The composition of $g_k$ and $\pi$ is
\[
 \pi\circ g_k(z,w)=
 \left[
  \frac{a_ke^{z+w}-a_{k+1}e^z}{a_{k-1}a_{k+1}-a_k^2}\,:\,
  \frac{a_{k-1}e^{z+w}-a_ke^z}{a_{k-1}a_{k+1}-a_k^2}
 \right]
\]
for $(z,w)\in\C_k\times\H_k$.
Hence
\[
 \chi_n(w)=\left[a_ke^w-a_{k+1}:a_{k-1}e^w-a_k\right]
\]
for $w\in\H_k$.
Note that $\chi_n(w)=[1:e^w]$ if $w\in\H_0$.

\begin{proof}
[Proof of Theorem \ref{main1}]
\underline{$n=1$.} Note that
\begin{align*}
 g_0(z_0,w_0)&=(e^{z_0},e^{z_0+w_0}) \\
 g_1(z_1,w_1)&=(e^{z_1}-e^{z_1+w_1},e^{z_1}) \\
 g_2(z_2,w_2)&=(-e^{z_2+w_2},e^{z_2}-e^{z_2+w_2})
\end{align*}
where $g_k:\C_k\times\H_k\rightarrow\C^2$.
Hence
\[
 \chi_1(w)=
 \begin{cases}
  [1:e^w] & (\text{if}\ w\in\H_0) \\
  [1-e^w:1] & (\text{if}\ w\in\H_1) \\
  [e^w:e^w-1] & (\text{if}\ w\in\H_2)\,.
 \end{cases}
\]

The projective special linear group $\mathrm{PSL}(2,\C)$ acts on
$\CP^1$.
Let $G_1=\begin{bmatrix} 0&1 \\ -1&1 \end{bmatrix}\in\mathrm{PSL}(2,\C)$,
of which order is $3$.
It is clear that
\[
 G_1\cdot[1:e^w]=[e^w:e^w-1]
 \quad\text{and}\quad
 G_1^2\cdot[1:e^w]=[1-e^w:1],
\]
that is $\chi_1(\H_2)=G_1\cdot\chi_1(\H_0)$ and
$\chi_1(\H_1)=G_1^2\cdot\chi_1(\H_0)$.
It is also clear that
\[
 \begin{cases}
  G_1\cdot[1:0]=[0:1] \\
  G_1\cdot[0:1]=[1:1]
 \end{cases}
 \quad\text{and}\quad
 \begin{cases}
  G_1^2\cdot[1:0]=[1:1] \\
  G_1^2\cdot[0:1]=[1:0]
 \end{cases}.
\]
Therefore $\chi_1$ wraps each upper half plane $\H_k$ on $\CP^1$
such that
\begin{itemize}
 \item wraps $\H_0$ on $\CP^1$ around points $[0:1]$ and $[1:0]$,
 \item wraps $\H_1$ on $\CP^1$ around points $[1:1]$ and $[1:0]$, and
 \item wraps $\H_2$ on $\CP^1$ around points $[1:1]$ and $[0:1]$.
\end{itemize}
See Figure \ref{chi1} for detail.
Moreover it is immediately follows that $\chi_1$ is a covering map if it
is restricted to the inverse image of
\begin{equation}
\label{cp1chi1}
 \mathcal{CP}_1=\CP^1\setminus\{[1:0],\,[0:1],\,[1:1]\}.
\end{equation}

\underline{$n=2$.} Since $a_k=k$,
\[
 \chi_2(w)=[ke^w-(k+1):(k-1)e^w-k]
\]
if $w\in\H_k$.
Let $G_2=\begin{bmatrix} 0&1 \\ -1&2 \end{bmatrix}\in\mathrm{PSL}(2,\C)$,
the power of which is
\[
 G_2^k=
 \begin{bmatrix} -k+1&k \\ -k&k+1 \end{bmatrix}
\]
for each $k\in\Z$.
It is clear that
\begin{align*}
 G_2^{-k}\cdot [1:e^w]
 &=[(k+1)-ke^w:k+(-k+1)e^w] \\
 &=[ke^w-(k+1):(k-1)e^w-k],
\end{align*}
that is $\chi_2(\H_k)=G_2^{-k}\cdot\chi_2(\H_0)$.
It is also clear that
\[
 \begin{cases}
  G_2^{-k}\cdot[1:0]=[k+1:k] \quad \text{and} \\
  G_2^{-k}\cdot[0:1]=[-k:-k+1]=[k:k-1].
 \end{cases}
\]
Therefore $\chi_2$ wraps
\begin{itemize}
 \item $\H_0$ on $\CP^1$ around points $[0:1]$ and $[1:0]$ and
 \item $\H_k$ on $\CP^1$ around points $[k:k-1]$ and $[k+1:k]$.
\end{itemize}
(Figure~\ref{chi2}.) Note that
\begin{itemize}
 \item
$\displaystyle\lim_{k\to\infty}[k+1:k]=\lim_{k\to\infty}[k:k+1]=[1:1]$
      and
 \item $[1:1]$ is the fixed point.
\end{itemize}
It is immediately that $\chi_2$ is a covering map if it is restricted to
\begin{equation}
\label{cp1chi2}
 \mathcal{CP}_2=\CP^1\setminus\{[1:1],\,[k:k+1]\,|\,k\in\Z\}.
\end{equation}

\underline{$n>2$.} Let
$G_n=\begin{bmatrix} 0&1 \\ -1&n \end{bmatrix}\in\mathrm{PSL}(2,\C)$,
the power of which is
\[
 G_n^k=
 \begin{bmatrix} -a_{k-1}&a_k \\ -a_k&a_{k+1} \end{bmatrix}
\]
for all $k\in\Z$.
It is clear that
\[
 G_n^{-k}\cdot[1:e^w]=[a_ke^w-a_{k+1}:a_{k-1}e^w-a_k],
\]
that is $\chi_n(\H_k)=G_n^{-k}\cdot\chi_n(\H_0)$.
It is also clear that
\[
 \begin{cases}
  G_n^{-k}\cdot[1:0]=[a_{k+1}:a_k]
  \quad \text{and} \\
  G_n^{-k}\cdot[0:1]=[a_k:a_{k-1}].
 \end{cases}
\]
Therefore $\chi_n$ wraps
\begin{itemize}
 \item $\H_0$ on $\CP^1$ around points $[0:1]$ and $[1:0]$ and
 \item $\H_k$ on $\CP^1$ around points $[a_k:a_{k-1}]$ and
       $[a_{k+1}:a_k]$.
\end{itemize}
(Figure~\ref{chi2}.) Note that, from (\ref{ak}),
\begin{itemize}
 \item
$\displaystyle\lim_{k\to\infty}[a_{k+1}:a_k]=%
\left[1:\frac{n-\sqrt{n^2-4}}{2}\right]$,
 \item
$\displaystyle\lim_{k\to\infty}[a_k:a_{k+1}]=%
\left[1:\frac{n+\sqrt{n^2-4}}{2}\right]$,
 \item $\displaystyle\left[1:\frac{n\pm\sqrt{n^2-4}}{2}\right]$ are the
       fixed points, and
 \item $G_n$ preserves the arc
$\displaystyle\left\{ [1:\lambda] \,\left|\, \frac{n-\sqrt{n^2-4}}{2}\leq%
\lambda\leq\frac{n+\sqrt{n^2-4}}{2}\right.\right\}$.
\end{itemize}
Let $\alpha=[1:\lambda]$ in the arc and let $U$ be any neighborhood of
$\alpha$ in $\CP^1$.
There is the connected component $V\subset\chi_n^{-1}(U)$ such that
$V\in\H_0$ and $V$ is the neighborhood of $\log\lambda$ in the boundary
of $\overline{\H_0}$.
The above notations show the $\chi_n|_V$ is not homeomorphism.
Though $\chi_n$ is a covering map if it is restricted to
\begin{equation}
\label{cp1chin}
 \mathcal{CP}_n=\CP^1\setminus\left\{
  \begin{array}{l}
   [a_k:a_{k+1}]\,\,(k\in\Z), \\
   \text{arc from %
   {\small $\left[1:\frac{n-\sqrt{n^2-4}}{2}\right]$} to
   {\small $\left[1:\frac{n+\sqrt{n^2-4}}{2}\right]$}}
  \end{array}
 \right\}
\end{equation}
\end{proof}

\begin{proof}
[Proof of Corollary~\ref{main2}]
 This is a corollary to Theorem \ref{main1}.
 The inverse images of (\ref{cp1chi1}), (\ref{cp1chi2}), and
 (\ref{cp1chin}) via $\pi:\C^2\setminus\{(0,0)\}\rightarrow\CP^1$ are
 given as follows:
 \begin{align*}
  \pi^{-1}(\mathcal{CP}_1)&=
  \C^2\setminus\{z_1=0\}\cup\{z_2=0\}\cup\{z_1=z_2\}, \\
  \pi^{-1}(\mathcal{CP}_2)&=
  \C^2\setminus
  \{z_2=z_1\}\cup\{kz_2=(k+1)z_1\,|\,k\in\Z\}, \\
  \pi^{-1}(\mathcal{CP}_n)&=
  \C^2\setminus
  \{a_kz_2=a_{k+1}z_1\,|\,k\in\Z\} \\
  &\qquad\qquad \cup
  \left\{z_2=\lambda z_1\,\left|\,\lambda\in\R,\,
  \frac{n-\sqrt{n^2-4}}{2}\leq\lambda\leq\frac{n+\sqrt{n^2-4}}{2}
  \right.\right\}.
 \end{align*}
 Then the map $\z:\stab(P_n)\rightarrow\C^2$ is a covering map if
 it is restricted to each of these subspaces.
\end{proof}

\bibliographystyle{alpha}

\end{document}